\documentclass{siamltex}
\usepackage{amsmath}
\usepackage{amssymb, color}
\usepackage{cite}

\usepackage{url}
\usepackage{hyperref}

\usepackage{subfigure}
\usepackage{graphicx}

\newtheorem{remark}[theorem]{Remark}

\newtheorem{assumption}{Assumption}

\definecolor{daocolor}{rgb}{0.0,0,1.0}

\newcommand{\R}{\mathbb{R}}
\newcommand{\C}{\mathbb{C}}

\newcommand{\p}{\partial}

\newcommand{\calH}{\mathcal{H}}
\newcommand{\calA}{\mathcal{A}}
\newcommand{\calL}{\mathcal{L}}
\newcommand{\calN}{\mathcal{N}}

\newcommand{\LN}{\left\|}
\newcommand{\RN}{\right\|}

\newcommand{\LC}{\left(}
\newcommand{\RC}{\right)}
\newcommand{\LB}{\left[}
\newcommand{\RB}{\right]}

\newcommand{\abs}[1]{\left\vert#1\right\vert}
\newcommand{\norm}[1]{\left\Vert#1\right\Vert}
\newcommand{\paren}[1]{\left(#1\right)}
\newcommand{\bracket}[1]{\left[#1\right]}
\newcommand{\set}[1]{\left\{#1\right\}}

\newcommand{\inner}[2]{\left\langle #1,#2\right\rangle}

\newcommand{\Xmax}{{X_{\max}}}
\newcommand{\Rmax}{{R_{\max}}}

\DeclareMathOperator{\littleo}{o}

\DeclareMathOperator{\cn}{cn}

\numberwithin{equation}{section}
\numberwithin{theorem}{section}

\title{Petviashvilli's Method for the Dirichlet Problem}

\author{D.~Olson\footnotemark[1] \and S.~Shukla\footnotemark[1]\and
  G.~Simpson\footnotemark[2] \and D.~Spirn\footnotemark[1]
\footnotetext[1]{School of Mathematics, University of Minnesota,
  Minneapolis, MN 55455, USA}
\footnotetext[2]{Department of Mathematics, Drexel University,
  Philadelphia, PA 19104, USA}}

\date{\today}

\begin{document}

\maketitle
\begin{abstract}
  We examine the Petviashvilli method for solving the equation
  $ \phi - \Delta \phi = |\phi|^{p-1} \phi$ on a bounded domain
  $\Omega \subset \R^d$ with Dirichlet boundary conditions.  We prove
  a local convergence result, using spectral analysis, akin to the
  result for the problem on $\R$ by Pelinovsky \& Stepanyants
  \cite{Pelinovsky:2004bv}.  We also prove a global convergence result by
  generating a suite of nonlinear inequalities for the iteration
  sequence, and we show that the sequence has a natural energy that
  decreases along the sequence.
\end{abstract}

\section{Introduction}

Many nonlinear dispersive wave equations, including the nonlinear
Schr\"odinger equation (NLS)
\begin{equation}
  \label{e:nls}
  i u_t + \Delta u + \abs{u}^{p-1} u = 0, \quad u:\R^{d+1} \to \C,
\end{equation}
and the nonlinear wave equation (NLW)
\begin{equation}
  \label{e:nlw}
  u_{tt} = \Delta u + \abs{u}^{p-1} u, \quad u:\R^{d+1}\to \C.
\end{equation}
possess exact nonlinear bound state solutions, such as solitons.
These can be obtained by first making the ansatz
$u(x, t) = e^{i \lambda t} \phi(x)$ with $\lambda \in \R$.  For
\eqref{e:nls}, one obtains
\begin{equation}
  \label{e:Rn_boundstate}
  -\lambda \phi + \Delta \phi  + \abs{\phi}^{p-1}\phi =0,\quad \phi:
  \R^{d} \to \C.
\end{equation}
A similar equation is obtained for \eqref{e:nlw}, but with $\lambda^2$
in place of $\lambda$.  Typically, we also assume
\begin{equation}
  \label{e:Rn_bc}
  \lim_{\abs{x}\to \infty} \abs{\phi(x) } = 0,
\end{equation}
though other boundary conditions are possible.  A challenge to
computing numerical solutions to \eqref{e:Rn_boundstate} such that
\eqref{e:Rn_bc} holds is that $\phi = 0$ is a solution.  Indeed, one
could attempt to solve \eqref{e:Rn_boundstate} by discretizing it and
then applying Newton's method. But its success would depend on the
initial guess, $u_0$, which could be in the domain of attraction of
the zero solution.  Thus, there is demand for an algorithm which can
avoid the trivial solution without requiring preconditioning of a
starting guess.

The same concerns arise when we consider the prototypical semilinear
elliptic equation
\begin{gather}
  \label{e:boundstate}
  -\phi + \Delta \phi + \abs{\phi}^{p-1}\phi=0, \quad
  \phi:\Omega \to \R\\
  \label{e:dirichlet_bc}
  \phi|_{\partial \Omega} = 0,
\end{gather}
where $\Omega$ is a bounded subset of $\R^d$ with sufficiently smooth
boundary.  This problem appears when one tries to numerically solve
\eqref{e:Rn_boundstate} subject to \eqref{e:Rn_bc}.  First the problem
must be reformulated on a finite domain, $\Omega$, and then an
artificial boundary condition must be introduced.  Here, we consider the Dirichlet boundary condition,
\eqref{e:dirichlet_bc}.  Inspired by work on \eqref{e:Rn_boundstate}
in \cite{Petviashvilli:1976aa}, we shall adapt and analyze
Petviashvilli's method, discussed below, to the Dirichlet problem.

\subsection{Ground States and Excited States}

Associated with \eqref{e:boundstate} is the energy functional
\begin{equation} \label{e:egroundstate} E(\psi) := {1\over 2} \LN \psi
  \RN^2_{H^1} - {1\over p+1} \LN \psi \RN^{p+1}_{L^{p+1}}.
\end{equation}
Taking the first variation of \eqref{e:egroundstate}, we see that
critical points of $E$ correspond to weak solutions, in the sense of
$H^1$, of \eqref{e:boundstate}. In general, this functional is
unbounded from below, so solutions of \eqref{e:boundstate} do not
correspond to global minimizers.  However, if one considers the
question of minimizing the energy over the set of nontrivial solutions
to \eqref{e:boundstate}, one finds that there is indeed a
minimizer---the so called ground state
\cite{Badiale:2011ug,Cazenave:aa,Sulem:1999kx}.

In general, one cannot rule out the existence of other nontrivial
solutions to \eqref{e:boundstate}.  Indeed, we immediately have that
if $u$ is a solution, so is $-u$.  But more complicated solutions may
also arise.  For the problem posed on all of $\R^d$, there are
infinitely many other solutions \cite{Sulem:1999kx}.  As these have
greater energy than the ground state, they are deemed to be excited
states.

Here, we will develop an algorithm which may be of use in computing
both ground and excited state solutions.

\subsection{Petviashvilli's Method \& Related Algorithms}

Originally developed to obtain soliton solutions of a
Kadomtsev--Petviashvilli equation, Petviashvilli's iteration method
for \eqref{e:boundstate} is
\begin{subequations}
  \label{e:petvia_iter}
  \begin{align}
    u_{n+1}& = M[u_n]^\gamma (I
             -\Delta)^{-1} \paren{\abs{u_n}^{p-1}u_n}, \\
    \label{e:Mfunc}
    M[u] &= \frac{\inner{(I - \Delta)u}{u}}{\inner{\abs{u}^{p-1}u}{u}},
  \end{align}
\end{subequations}
where the inner product is that of $L^2$, and $\gamma$ is in an
admissible range identified below in Assumption \ref{a:gamma}. For
brevity of notation, we denote the iteration operator for
\eqref{e:petvia_iter} by $\calA$:
\begin{equation*}
  \label{e:petvia_op}
  u_{n+1} = \calA(u_n).
\end{equation*}

As was shown in \cite{Pelinovsky:2004bv}, this method is locally
non-linearly convergent about nonlinear bound states.  There, the
result was demonstrated for the case of $d=1$ and on the real line but
allowed for more general elliptic operators. Analysis of the algorithm
to more specific cases appeared in \cite{Chugunova:2007wg} for a fifth
order Korteweg--de~Vries equation and in \cite{Demanet:2006gi} for
cubic NLS in $\R^3$.  In practice, this algorithm is extremely
robust to the choice of $u_0$.  Initial conditions, far from any
solution, still converge to solutions of \eqref{e:boundstate}; see the
numerical experiments, below, in Section \ref{s:examples}.  This
robustness, which is not addressed by local analysis, is one of the
main motivations for our work.

Petviashvilli's method was also generalized in
\cite{Lakoba:2007cg,Musslimani:2004wx} allowing for external
potentials and more complicated ({\it i.e.} non-homogeneous)
nonlinearities.  In these works, they were able to handle semilinear
elliptic systems and also obtained excited state solutions.  The
performance of the algorithm was also improved using ``mode
elimination'' \cite{Lakoba:2007kp}.  More recently, the $M$
functional, \eqref{e:Mfunc}, was generalized in \cite{Alvarez:2014ec}.

Algorithms related to Petviashvilli's method include the imaginary
time method \cite{Yang:2008kn}, the squared operator method
\cite{Yang:2007vu}, and spectral renormalization
\cite{Ablowitz:2005tu}.  The imaginary time method allows the user to
find a solution of \eqref{e:Rn_boundstate} with a given $L^2$ norm,
but unknown constant $\lambda$.  This if of use in certain physical
applications such as Bose-Einstein condensate.  The squared operator
method improves the performance and stability of these
algorithms. Lastly, spectral renormalization allows for more general
nonlinearities and external potentials.

Numerically, solutions are typically sought for the problem posed on
$\R$ or $\R^2$ by studying the problem on a sufficiently large finite
domain, $[-x_{\max}, x_{\max})^d$, with periodic boundary conditions.
This allows for the use of the Fast Fourier Transform, making the
inversion of constant coefficient elliptic operators easy.  In at
least one work, where a radially symmetric solution was sought, the
domain was reduced to $[0, \infty)$, and an artificial Robin boundary
condition was imposed at some $r_{\max}$ \cite{Baruch:2011fm}.

\subsection{Main Results \& Relations to Other Work}

We consider the Dirichlet problem, \eqref{e:boundstate}, and apply
iteration scheme \eqref{e:petvia_iter} to obtain non-trivial
solutions.  In this context, we are able to obtain a strong local
convergence result and a weak global convergence result.

For both results, we make the following assumptions:
\begin{assumption}
  \label{a:domain}
  $\Omega\subset \R^d$ is a bounded, convex,
  open set (hence, it has Lipschitz boundary).
\end{assumption}
\begin{assumption}
  \label{a:subcrit}
  The nonlinearity is superlinear and subcritical, satisfying:
  \begin{equation} \label{e:subcrit} 1< p < d^\dagger
    \equiv \begin{cases}
      \infty & d=1,2,\\
      \frac{d+2}{d-2} & d \geq 3.
    \end{cases}
  \end{equation}
\end{assumption}
Next, we assume that $\gamma$ lives in an admissible range, identified
in \cite{Pelinovsky:2004bv},
\begin{assumption}
  \label{a:gamma}
  \begin{equation}
    \label{e:gamma_range}
    1 < \gamma <\frac{p+1}{p-1}.
  \end{equation}
\end{assumption}

Assumption \ref{a:domain} is essential to our use of elliptic theory
and Sobolev embedding results.  For computation, this is a very mild
restriction, as our domains are typically polygonal or smoother.
Assumption \ref{a:subcrit} is needed to ensure the existence of
non-trivial solutions to \eqref{e:boundstate} and will be used
throughout the paper.  Assumption \ref{a:gamma} is needed for the
algorithm to be linearly stable.  Within the range of admissible
$\gamma$, a distinguished value is
\begin{equation}
  \label{e:gamma_star}
  \gamma_\star = \frac{p}{p-1}.
\end{equation}
This will play a special role in both our linear and nonlinear
analysis.

To state our first result, we will need one additional, spectral,
assumption.  We note here, and highlight again in Section
\ref{s:spectral}, that $p$ is an eigenvalue of
$(I-\Delta)^{-1}(p\abs{\phi}^{p-1}\bullet)$ corresponding to the
eigenfunction $\phi$.  We assume:
\begin{assumption}
  \label{a:spectral}
  Viewing $(I-\Delta)^{-1}(p\abs{\phi}^{p-1}\bullet)$ as an operator
  from $H^1_0$ to itself,
  \begin{equation}
    \label{e:spec_cond}
    \sigma((I-\Delta)^{-1}(p\abs{\phi}^{p-1}\bullet))\setminus (-\infty, 1) = \set{p},
  \end{equation}
  and the algebraic multiplicity of $p$ is one.
\end{assumption}

Our first result is on the local convergence of the algorithm.
\begin{theorem}[Strong Local Convergence]
  \label{t:local}
  Assume Assumptions \ref{a:domain}, \ref{a:subcrit}, \ref{a:gamma},
  and \ref{a:spectral} hold.  Given a solution $\phi$ of
  \eqref{e:boundstate}, let
  $\calA'(\phi):H^1_0(\Omega) \to H^1_0(\Omega)$ denote the
  linearization of \eqref{e:petvia_iter} about $\phi$.

  There exists a neighborhood, $\mathcal{N}$, of $\phi$ in $H^1_0$ and
  a constant $\theta\in (0,1)$ such that for all
  $u_0 \in \mathcal{N}$, the sequence of iterates,
  $\{u_n\}_{n=0}^\infty$, defined through \eqref{e:petvia_iter}
  satisfies
  \begin{equation*}
    \norm{u_n - \phi}_{H^1} \leq \theta^n \norm{u_0 - \phi}_{H^1}.
  \end{equation*}
\end{theorem}

Our second main result is:

\begin{theorem}[Subsequential Global Convergence]
  \label{t:global}

  Assume Assumptions \ref{a:domain} and \ref{a:subcrit} hold and that
  $u_0$ is any nontrivial $H^1_0$ function.  Also assume that
  $\gamma =\gamma_\star$, defined in \eqref{e:gamma_star}.

  Let $\{u_n\}_{n=0}^\infty$ be a sequence of functions generated by
  \eqref{e:petvia_iter}.

  \begin{itemize}

  \item {Assume $d\geq 2$.  There exists a subsequence $\{u_{n_k}\}$
      that converges strongly in $H^1_0$ to a nontrivial strong
      solution of \eqref{e:boundstate} in $W^{2, 2}$ for $d=2$ and to
      a strong solution in $W^{2,\kappa}$ for $d \geq 3$ where
      $\kappa = \min\left\{2, \frac{2^*}{p}\right\}$. The Sobolev
      conjugate, $2^*$, is defined as $2^* = \frac{2d}{d-2}$.}

  \item Assume either $d=1$ or $d\geq 2$ with $\partial \Omega$ is
    smooth. Then, for all $\ell \in 2\mathbb{Z}^+$, there exists a
    strongly convergent subsequence in $C^\ell$, and its limit,
    $\phi\in C^\infty$, is a nontrivial solution of
    \eqref{e:boundstate}.

  \end{itemize}

\end{theorem}

Combining these two results, we have:
\begin{corollary}
  \label{c:combo}
  If the limit obtained in Theorem \ref{t:global} satisfies Assumption
  \ref{a:spectral}, then the entire sequence converges to $\phi$.
\end{corollary}

Thus, we have a partial answer towards the question of the robustness
of the algorithm with respect ot the choice of $u_0$.

The proof of Theorem \ref{t:local} is based on the result in
\cite{Pelinovsky:2004bv} for the problem on $\R$.  See
\cite{Chugunova:2007wg,Demanet:2006gi,Lakoba:2007cg} for related
analysis of the algorithm.  Theorem \ref{t:global} is proven by
compactness methods, and we believe it to be novel.

\subsection{Outline}

Our paper is organized as follows.  In Section \ref{s:prelim}, we
review some preliminary results on the existence and regularity of
nonlinear bound states, the well-posedness of the iteration scheme,
and important spectral properties {of the iteration operator}.  In
Section \ref{s:local}, we prove Theorem \ref{t:local}, and a proof of
Theorem \ref{t:global} is given in Section \ref{s:global}.  Example
computations appear in Section \ref{s:examples} followed by a
discussion in Section \ref{s:disc}.

\section{Preliminary Results}
\label{s:prelim}

In this section, we provide some definitions, review some properties
of the bound states, and examine the iteration scheme.  For brevity,
we define
\begin{equation*}
  \label{e:calLdef}
  \calL \equiv  I - \Delta.
\end{equation*}
As an inner product on $H^1_0(\Omega)$, we use
\begin{equation*}
  \label{e:h1inner}
  \inner{f}{g}_{H^1} = \inner{\calL f}{g} =\inner{f}{\calL g} = \int_\Omega  fg +
  \nabla f \cdot \nabla g.
\end{equation*} 
Weak solutions of \eqref{e:boundstate} are functions,
$\phi\in H^1_0(\Omega)$ such that for all $v \in H^1_0$,
\begin{equation*}
  \label{e:weaksoln}
  \int_\Omega \nabla \phi \cdot \nabla v + \phi v +  \abs{\phi}^{p-1}\phi v = 0
\end{equation*}
Relating this to the iteration scheme, we have the following
elementary result:
\begin{lemma}
  If $u\in H^1_0$ is a fixed point of \eqref{e:petvia_iter} and
  $M[u]=1$, then $u$ is a weak solution of \eqref{e:boundstate}.
\end{lemma}
\begin{proof}
  Since $\calL :H^1_0 \to H^{-1}$ has trivial kernel, $\calL$ can be
  applied to both sides of \eqref{e:petvia_iter} to obtain
  $\calL u = \abs{u}^{p-1}u$, with equality holding in $H^{-1}$.
\end{proof}

Existence of weak, nontrivial solutions subject to Assumptions
\ref{a:domain} and \ref{a:subcrit} is well established and can be
found in, for example, \cite{Cazenave:aa,Badiale:2011ug}.  If we
assume that $\partial \Omega$ is $C^\infty$, then, using bootstrap
methods and elliptic regularity theory, $\phi$ will also be smooth.

\subsection{Properties of the Iteration Method}

Before proceeding with an analysis of the algorithm, we need to
establish certain properties.

\begin{proposition}
  \label{p:wellposed}
  Assume Assumptions \ref{a:domain} and \ref{a:subcrit} hold.  If
  $u \in H^1_0$ and $u\neq 0$, then $\calA(u) \in H^1_0$ and
  $\calA(u) \neq 0$.
\end{proposition}
\begin{proof}
  By the standard Sobolev embeddings, since $u \in H^1_0$,
  $u \in L^{p+1}$.  Therefore the numerator and denominator of $M[u]$
  are both finite.  In addition, since $u \neq 0$, the denominator is
  nonzero.  Hence, $M[u]$ is well defined.  Next, observe that
  $\abs{u}^{p-1} u \in H^{-1}$. Finally, we have that
  $\calL^{-1}(\abs{u}^{p-1} u) \in H^1_0$; hence,
  $\calA(u) \in H^1_0$.  Since $M[u] >0$, we must have that
  $\calA(u) \neq 0$.
\end{proof}

This implies that the iteration scheme is well defined in the sense
that, though $\norm{u_n}_{H^1}$ could tend to infinity as
$n\to \infty$, $u_{n+1}\in H^1_0$ can always be computed from $u_n$.

\begin{proposition}
\label{p:lin_op}
  Assume Assumptions \ref{a:domain} and \ref{a:subcrit} hold and that
  $\gamma>1$.  Then about any $u \in H^1_0$, $u\neq 0$:
  \begin{itemize}
  \item $\calA$ is continuous;
  \item $\calA$ is Fr\'echet differentiable;
  \item The Fr\'echet derivative, $\calA'(\bullet):H^1_0\to L(H_0^1)$,
    is continuous with respect to the $H^1$ operator norm.
  \end{itemize}
  The Fr\'echet derivative is given by:
  \begin{subequations}
    \label{e:frechet}
    \begin{align*}
      \calA'(u)h &= M[u]^\gamma\calL^{-1} ({p \abs{u}^{p-1} h}) +
                   \gamma
                   \frac{[M'[u],h]}{M[u]}\calA(u),\\
      [M'[u],h] & = \frac{1}{\norm{u}_{L^{p+1}}^{{p+1}}}\bracket{2
                  \inner{u}{h}_{H^1}- (p+1) M[u] \inner{\abs{u}^{p-1}u}{h}}.
    \end{align*}
  \end{subequations}
\end{proposition}
\begin{proof}
  The proof of this is standard, and, for brevity, we omit it.
\end{proof}

\subsection{Spectral Results}
\label{s:spectral}

At a nontrivial solution $\phi$ of \eqref{e:boundstate},
$M[\phi] = 1$, $\calL \phi = \abs{\phi}^{p-1}\phi$, and
$\calA(\phi) = \phi$.  Therefore the linearization at $\phi$
simplifies to:
\begin{equation}
  \label{e:frechetphi2}
  \calA'(\phi)h = \calL^{-1}(p\abs{\phi}^{p-1} h) - \gamma(p-1) \frac{\inner{\phi}{h}_{H^1}}{\inner{\phi}{\phi}_{H^1}}\phi.
\end{equation}
Equation \eqref{e:frechetphi2} is quite informative as to the local
behavior of the algorithm.  Returning to \eqref{e:boundstate}, one
might ask why not apply the iteration scheme
\begin{equation}
  \label{e:naiveiter}
  u_{n+1} =\calL^{-1}(\abs{u_n}^{p-1}u_n).
\end{equation}
Linearizing \eqref{e:naiveiter} about $\phi$, we obtain
\[
\calL^{-1}(p \abs{\phi}^{p-1} h),
\]
which is the first term in \eqref{e:frechetphi2}.  Observe that this
has an eigenvalue with modulus in excess of unity:
\[
\calL^{-1}(p \abs{\phi}^{p-1} \phi) = p \phi.
\]
Thus, there is a linearly unstable mode, and \eqref{e:naiveiter} will
diverge.  In \eqref{e:frechetphi2}, a rank one orthogonal projection
is added, shifting the unstable eigenvalue into the interior of the
unit disk.  Indeed,
\begin{equation}
  \calA'(\phi)\phi = p \phi - \gamma(p-1)\phi = \paren{p - \gamma(p-1)}\phi.
\end{equation}
Assumption \ref{a:gamma} is precisely what ensures that the eigenvalue
corresponding to $\phi$ lies in the interval
$-1 < p - \gamma(p-1) < 1$.

Next, we have:
\begin{proposition}
  \label{p:compactness}
  $\calL^{-1}(p\abs{\phi}^{p-1}\bullet):H^1_0\to H^1_0$ is compact.
\end{proposition}
\begin{proof}
  Given any bounded sequence $\set{u_n}$ in $H^1_0$, it contains a
  weakly convergent subsequence, $\set{u_{n_k}}$, with limit
  $u\in H^1_0$.  Thus, to show that the operator is compact, it
  suffices to show that it maps weakly convergent $H^1$ sequences to
  strongly convergent ones.  Consider the case of $d\geq 3$, as
  $d=1,2$ are simpler.  First, we have
  \begin{equation*}
    \begin{split}
      &\norm{\calL^{-1}(p\abs{\phi}^{p-1}(u_{n_k} - u))}_{H^1}^2=
      \inner{\calL^{-1}(p\abs{\phi}^{p-1}(u_{n_k} -
        u))}{(p\abs{\phi}^{p-1}(u_{n_k} - u))}\\
      &\leq \norm{\calL^{-1}(p\abs{\phi}^{p-1}(u_{n_k} -
        u))}_{L^{2^*}}\norm{p\abs{\phi}^{p-1}}_{L^{2^* \over p-1}}
      \norm{u_{n_k} - u}_{L^{2^* \over 2^* - p}}.
    \end{split}
  \end{equation*}
  Since $H^1$ continuously embeds into $L^{2^*}$ and
  $L^{{2^*}\over{p-1}}$, the first two norms in the above expression
  are controlled by $H^1$.  Next, since $p$ is subcritical, satisfying
  \eqref{e:subcrit}, we have that ${2^*}/{(2^* - p)} < 2^*$;
  therefore,
  \[
  \norm{\calL^{-1}(p\abs{\phi}^{p-1}(u_{n_k} - u))}_{H^1}\leq C
  \norm{u_{n_k} - u}_{L^{2^* \over 2^* - p}}.
  \]
  By Rellich, the subsequence converges strongly in the
  $L^{2^* \over 2^* - p}$ topology, and we have the desired
  convergence.
\end{proof}

We can now summarize the spectrum of the operator
$\calL^{-1}(p\abs{\phi}^{p-1}\bullet)$:
\begin{theorem}[Spectrum of
  $\calL^{-1}(p\abs{\phi}^{p-1}\bullet)$]\label{t:spec}\quad
  \begin{enumerate}
  \item $\calL^{-1}(p\abs{\phi}^{p-1}\bullet)$ is self-adjoint with
    respect to the $H^1$ inner product and the spectrum is real.
  \item $\calL^{-1}(p\abs{\phi}^{p-1}\bullet):H^1_0\to H^1_0$ has a
    complete orthonormal basis, $\left\{\psi_j\right\}$, with
    $\calL^{-1}(p\abs{\phi}^{p-1} \psi_j) = \nu_j \psi_j$ and
    $\nu_j \to 0$ and $\nu_j\geq 0$.
  \item $\calL^{-1}(p\abs{\phi}^{p-1}\phi) = p\phi$.
  \end{enumerate}
\end{theorem}

\begin{proof}
  Self-adjointness follows by inspection, and this immediately implies
  that the spectrum must be real.  Since it is compact on $H^1_0$ and
  self-adjoint on the separable Hilbert space $H^1_0$, it has a
  complete orthonormal basis, and all nonzero eigenvalues have finite
  multiplicity; see Theorems VI.15 and VI.16 of \cite{Reed:1980aa}.
  As previously noted, $p$ is an eigenvalue with $\phi$ the
  corresponding eigenfunction.  Lastly, given any $\nu_j$, we compute,
  \begin{equation*}
    \nu_j = \nu_j \inner{\psi_j}{\psi_j}_{H^1} =
    \inner{\calL^{-1}(p\abs{\phi}^{p-1}
      \psi_j)}{\psi_j}_{H^1}= \inner{p\abs{\phi}^{p-1} \psi_j}{\psi_j} \geq 0.
  \end{equation*}
\end{proof}

We then immediately obtain results on the spectrum of
$\calA'(\phi)\bullet$:
\begin{corollary}
  $\calA'(\phi)\bullet: H^1_0 \to H^1_0$ is a compact.
\end{corollary}
\begin{proof}
  Since $\calA'(\phi)\bullet$ is a compact operator plus rank one
  operator, it is compact.
\end{proof}

\begin{corollary}[Spectrum of $\calA'(\phi)$]\label{c:Aspec}\quad

  \begin{enumerate}
  \item $\calA'(\phi) \bullet$ is self adjoint with respect to the
    $H^1$ inner product, and the spectrum of $\calA'(\phi)\bullet$ is
    real.
  \item $\calA'(\phi) \bullet:H^1_0 \to H^1_0$ has a complete
    orthonormal basis, $\left\{\psi_j\right\}$, with
    $\calA'(\phi) \psi_j = \mu_j\psi_j$ and $\mu_j \to 0$.
  \item $\calA'(\phi)\phi = (p - (p-1)\gamma) \phi$.
  \item For all $\mu_j \neq p - (p-1)\gamma$, $\mu_j \geq 0$.
  \end{enumerate}

\end{corollary}
Assumption \ref{a:spectral} thus ensures that $p$ is the only linearly
unstable mode of $\calL^{-1} (p \abs{\phi}^{p-1}\bullet)$, which is
linearly stable for $\calA'(\phi) \bullet$.

\begin{proof}
  All of this follows from the last corollary and Theorem
  \ref{t:spec}.  The eigenvalues $\nu_j$ and $\mu_j$ relate to one
  another as
  \[
  \mu_j =\begin{cases}
    \nu_j,& \nu_j \neq p,\\
    p-(p-1)\gamma, & \nu_j = p.
  \end{cases}
  \]
  This follows from $\calA'(\phi)\bullet$ being a rank one
  perturbation, where the rank one term is an orthogonal spectral
  projection.

\end{proof}

\section{Local Convergence}
\label{s:local}

Given $v_0 \in H^1_0$, let $v_n$ denote the sequence generated by the
linearized operator:
\begin{equation}
  \label{e:linear_iter}
  v_{n+1} = \calA'(\phi) v_n.
\end{equation}
First, we decompose the iterates.
\begin{proposition}
  The sequence $v_n$ can be decomposed as $v_n = a_n \phi + w_n$,
  where $w_n \perp_{H^1} \phi = 0$, and $a_n$ and $w_n$ satisfy the
  decoupled equations:
  \begin{subequations}
    \begin{align}
      \label{e:wdecomp}
      w_{n+1} & = \calL^{-1}(p\abs{\phi}^{p-1} w_n)\\
      \label{e:adecomp}
      a_{n+1} & = (p - \gamma(p-1))a_n.
    \end{align}
  \end{subequations}
\end{proposition}
\begin{remark}
  Notice in the case that $\gamma = \gamma_\star = \tfrac{p}{p-1}$,
  \eqref{e:adecomp} implies that $a_n = 0$ for all $n \geq 1$,
  regardless of the choice of $a_0$.
\end{remark}

\begin{proof}
  This follows from the spectral decomposition of the operator.
  Indeed, at any iterate, we can define
  \[
  a_n = \frac{\inner{\phi}{v_n}_{H^1}}{\inner{\phi}{\phi}_{H^1}}
  \]
  and then let $w_n = v_n - a_n \phi$.  Thus,
  $w_n \perp_{{H^1}} \phi$.  Comparing sequential iterates, we see
  \begin{equation*}
    \begin{split}
      v_{n+1} = \calA'(\phi)v_n& = \calA'(\phi)w_n  + a_n \calA'(\phi)\phi,\\
      w_{n+1} + a_{n+1} \phi& = \calL^{-1}(p\abs{\phi}^{p-1} w_n) +
      (p-\gamma(p-1))a_n \phi.
    \end{split}
  \end{equation*}
  Taking the $H^1$ inner product of both sides with $\phi$, we have
  $a_{n+1}=(p-\gamma(p-1))a_n$.  Going back and cancelling this out in
  the equation above yields
  $w_{n+1} = \calL^{-1}(p\abs{\phi}^{p-1} w_n)$.

\end{proof}

\begin{proposition}
  \label{p:lineariter}
  Assuming that Assumptions \ref{a:domain},\ref{a:subcrit},
  \ref{a:gamma}, and \ref{a:spectral} hold, let
  \begin{equation*}
    \label{e:mustar}
    \mu_\star = \sup \sigma(\calL^{-1}(p\abs{\phi}^{p-1}\bullet))\setminus
    \set{p}.
  \end{equation*}
  Then $\mu_\star \in [0,1)$, and for sequence \eqref{e:linear_iter}
  \begin{equation}
    \label{e:linear_conv}
    \norm{v_n}_{H^1}\leq \max\set{\abs{p- \gamma(p-1)}, \mu_\star}^n \norm{v_0}_{H^1}.
  \end{equation}
\end{proposition}

\begin{remark}
  Again, $\gamma=\gamma_\star$ plays a distinguished role in the above
  contraction estimate, making the constant $\mu_\star^n$.
\end{remark}

\begin{proof}
  By virtue of spectral assumption \eqref{e:spec_cond} and the
  compactness of $\calL^{-1}(p\abs{\phi}^{p-1}\bullet)$, we are
  assured that $\mu_\star<1$.  This is because the only possible
  cluster point in the spectrum of a compact operator is zero.  The
  fact that $\mu_\star \geq 0$ follows from Theorem \ref{t:spec}.

  Writing $v_{n+1} = w_{n+1} + a_{n+1} \phi$ and recalling
  $w_n \perp_{{H^1}} \phi$,
  \begin{equation*}
    \begin{split}
      \norm{v_{n+1}}_{H^1}^2 &= \norm{w_{n+1}}_{H^1}^2 + 2 a_{n+1}
      \inner{w_{n+1}}{\phi}_{H^1} + a_{n+1}^2 \norm{\phi}_{H^1}^2\\
      & = \norm{w_{n+1}}_{H^1}^2 + a_{n+1}^2 \norm{\phi}_{H^1}^2.
    \end{split}
  \end{equation*}
  Since $w_n \perp_{H^1} \phi$ and since we have assumed $p$ is the
  only eigenvalue of $\calL^{-1}(p\abs{\phi}^{p-1}\cdot)$ greater than
  or equal to one, then, by the spectral decomposition of the
  operator,
  \[
  w_{n+1} = \calL^{-1}(p\abs{\phi}^{p-1}w_n) = \sum_{\mu_j < 1} \mu_j
  \inner{\psi_j}{w_{n+1}}_{H^1} \psi_j.
  \]
  Therefore,
  \begin{equation*}
    \begin{split}
      \norm{w_{n+1}}_{H^1}^2 &= \sum_{\mu_j < 1} \abs{\mu_j}^2
      \abs{\inner{\psi_j}{w_{n+1}}_{H^1}}^2\leq \mu_\star^2
      \norm{w_{n+1}}_{H^1}^2,
    \end{split}
  \end{equation*}
  and
  \begin{equation*}
    \begin{split}
      \norm{v_{n+1}}_{H^1}^2 &\leq \mu_\star^2 \norm{w_{n+1}}_{H^1}^2+
      \abs{p -
        \gamma(p-1)}^2 a_n^2 \norm{\phi}_{H^1}^2\\
      &\leq \max\set{\mu_\star^2,\abs{p - \gamma(p-1)}^2
      }\norm{v_n}_{H^1}^2.
    \end{split}
  \end{equation*}

\end{proof}

Finally, we state our local nonlinear convergence result.
\begin{theorem}
  Under the same assumptions as in Proposition \ref{p:lineariter},
  further assume $\gamma$ satisfies \eqref{e:gamma_range}. Then there
  exists a neighborhood $\mathcal{N}$ of $\phi$ such that if
  $u_0 \in \mathcal{N}$, then $u_n\to \phi$.
\end{theorem}

\begin{proof}
  Let $\eta= \max\set{\abs{p- \gamma(p-1)}, \mu_\star}<1$.  Since
  $u\mapsto \calA'(u)$ is continuous in the operator norm (see
  Proposition \ref{p:lin_op}), there is a
  neighborhood $\calN$ of $\phi$, such that
  \begin{equation*}
    \begin{split}
      \sup_{u \in \calN} \norm{ \calA'(u)} &\leq \sup_{u \in \calN}
      \norm{ \calA'(u)-\calA'(\phi)} + \norm{\calA'(\phi)}\leq
      \frac{1}{2}(1-\eta)+ \eta = \frac{1+\eta}{2}<1.
    \end{split}
  \end{equation*}
  Let $\theta\equiv (1+\eta)/2$. By the mean value theorem for
  Fr\'echet differentiable functions, \cite{Hutson:2005aa}, for all
  $u,v \in \calN$,
  \begin{equation*}
    \begin{split}
      \norm{\calA(u) - \calA(v)}_{H^1} &\leq \sup_{w \in \calN} \norm{
        \calA'(w)}\norm{u-v}_{H^1}\leq \theta \norm{u-v}_{H^1}.
    \end{split}
  \end{equation*}
  Therefore, we have a contraction, with parameter $\theta$.

\end{proof}

\section{Global Convergence}
\label{s:global}

In this section we provide results concerning the global convergence
of $\{ u_n \}$ and prove Theorem \ref{t:global}.

\subsection{Nonlinear Estimates}

Before we can proceed, we need certain {\it a priori} estimates on our
sequence.  We emphasize that these estimates are inherently nonlinear,
in contrast to those used to prove Theorem \ref{t:local}.  We also
define the $\alpha$ exponent,
\begin{equation}
  \label{e:alphadef}
  \alpha \equiv \gamma -\gamma p + p,
\end{equation}
which will play an important role in what follows.  By Assumptions
\ref{a:subcrit} and \ref{a:gamma}, we are assured that
$\abs{\alpha}<1$.

First, we have the following sequential inequalities.
\begin{lemma}\label{l.ineqs}
  Assume Assumptions \ref{a:domain}, \ref{a:subcrit}, and
  \ref{a:gamma} hold.  Also assume that $u_0\in H^1$ is nontrivial.
  Then the following identities hold for the sequence generated by
  \eqref{e:petvia_iter} for all $n \geq 0$:

  \begin{subequations}
    \begin{align}
      M[u_n]^{\gamma -1 } & \leq { \LN u_{n+1} \RN_{H^1} \over \LN u_{n} \RN_{H^1} }, \label{e:Mgrowth} \\
      {\LN u_{n+1} \RN_{H^1} \over \LN u_{n+1} \RN_{L^{p+1}} } & \leq {\LN u_{n} \RN_{H^1} \over \LN u_{n} \RN_{L^{p+1}} },  \label{e:revSob} \\
      M[u_n]^{\gamma -1 }  & \geq { \LN u_{n+1} \RN^2_{H^1} \over \LN u_{n} \RN^2_{H^1} } {\LN u_{n} \RN_{L^{p+1}} \over \LN u_{n+1} \RN_{L^{p+1}} } \label{e:Mgrowth2}, \\
      M[u_{n+1}] & \leq M[u_n]^\alpha, \label{e:Mmonotone}
    \end{align}
  \end{subequations}
  with $\alpha$ defined as in \eqref{e:alphadef}.

\end{lemma}

\begin{proof}
  {\bf Step 1.} Applying $I-\Delta$ to both sides of
  \eqref{e:petvia_iter}, we have
  \begin{equation} \label{derivativeform} (I - \Delta ) u_{n+1} =
    M[u_n]^\gamma(|u_n|^{p-1}u_n).
  \end{equation}
  Multiplying by $u_{n}$ and integrating by parts, we obtain the
  discrete identity
  \begin{equation} \label{e:discident} { \inner{ u_{n+1}}
      {u_n}_{H^1}\over \| u_n \|_{H^1}^2} = M[u_n]^{\gamma-1}.
  \end{equation} 
  Applying Cauchy-Schwarz to the above expression, we have
  \eqref{e:Mgrowth}.

  \noindent {\bf Step 2.}  
  If we multiply \eqref{derivativeform} by $u_{n+1}$, integrate by
  parts, and then use H\"older's inequality we have
  \begin{equation}\label{e:ineq1}
    \LN u_{n+1} \RN^2_{H^1} \leq M[u_n]^\gamma \LN u_{n+1} \RN_{L^{p+1}} \LN u_n \RN_{L^{p+1}}^p.
  \end{equation}
  On the other hand, if we multiply \eqref{derivativeform} by $u_n$
  and use Cauchy-Schwarz on the $H^1$ inner product, we have
  \begin{equation}\label{e:ineq2}
    M[u_n]^\gamma \LN u_{n} \RN^{p+1}_{L^{p+1}} \leq \LN u_{n+1} \RN_{H^1} \LN u_n \RN_{H^1}.
  \end{equation}

  Next, multiplying \eqref{e:ineq1} by $\LN u_{n} \RN_{L^{p+1}}$ and
  using \eqref{e:ineq2}, we have
  \begin{equation*}
    \begin{split}
      \LN u_{n+1} \RN^2_{H^1} \LN u_{n} \RN_{L^{p+1}} & \leq M[
      u_n]^\gamma \LN u_n \RN_{L^{p+1}}^{p+1} \LN u_{n+1}
      \RN_{L^{p+1}} \leq \LN u_{n+1} \RN_{H^1} \LN u_n \RN_{H^1} \LN
      u_{n+1} \RN_{L^{p+1}} .
    \end{split}
  \end{equation*}
  Inductively applying Proposition \ref{p:wellposed}, we are assured
  that $u_n$ is always nontrivial, so we can divide both sides of the
  above expression by
  $\LN u_{n+1} \RN_{H^1} \LN u_{n+1} \RN_{L^{p+1}} \LN u_n
  \RN_{L^{p+1}}$ to get, by induction,
  \begin{equation} \label{monotonicity} { \LN u_{n+1} \RN_{H^1} \over
      \LN u_{n+1} \RN_{L^{p+1}} } \leq { \LN u_{n} \RN_{H^1} \over \LN
      u_{n} \RN_{L^{p+1}}} \leq \cdots \leq { \LN u_{0} \RN_{H^1}
      \over \LN u_{0} \RN_{L^{p+1}} } \equiv C_0.
  \end{equation}
  This is \eqref{e:revSob}, and we have a reverse Sobolev embedding
  along the sequence
  \begin{align} \label{reverseSobolev} \LN u_{n} \RN_{H^1} \leq C_0
    \LN u_{n} \RN_{L^{p+1}}.
  \end{align}
  Recall the standard Sobolev embedding
  \begin{equation} \label{Sobolev} \LN u_{n} \RN_{L^{p+1}} \leq C_p
    \LN u_{n} \RN_{H^1},
  \end{equation}
  for all $n \geq 0$ since $p < d^\dagger$.

  \noindent {\bf Step 3.}  To establish \eqref{e:Mmonotone}, we first
  note
  \begin{equation*} { M [u_{n+1} ] \over M [u_n] }
    = { \LN u_{n+1} \RN_{H^1}^2 \over \LN u_{n} \RN_{H^1}^2} { \LN
      u_{n} \RN_{L^{p+1}}^{p+1} \over \LN u_{n+1} \RN_{L^{p+1}}^{p+1}
    } \stackrel{\eqref{monotonicity}}{\leq} { \LN u_{n}
      \RN_{L^{p+1}}^{p-1} \over \LN u_{n+1} \RN_{L^{p+1}}^{p-1} },
  \end{equation*}
  so
  \begin{equation} \label{e:MLp} { \LN u_{n} \RN_{L^{p+1}} \over \LN
      u_{n+1} \RN_{L^{p+1}} } \geq \LC { M[u_{n+1}] \over M [u_n] }
    \RC^{1 \over p-1}.
  \end{equation}

  Next, squaring \eqref{e:ineq2} and dividing by
  $\LN u_{n+1} \RN_{L^{p+1}}^{p+1} \LN u_n \RN_{H^1}^2$, we get
  \begin{align*}
    M[u_{n+1}] & \geq M[u_n]^{2 \gamma} { \LN u_n \RN_{L^{p+1}}^{p+1}
                 \over \LN u_n \RN^2_{H^1}}
                 { \LN u_n \RN_{L^{p+1}}^{p+1} \over \LN u_{n+1} \RN_{L^{p+1}}^{p+1}  } \\
               & \quad = M[u_n]^{2 \gamma - 1}{ \LN u_n \RN_{L^{p+1}}^{p+1} \over
                 \LN u_{n+1} \RN_{L^{p+1}}^{p+1} } \stackrel{
                 \eqref{e:MLp}}{\geq} M[u_n]^{2 \gamma - 1} {
                 M[u_{n+1}]^{p+1\over p-1} \over M[u_{n}]^{p+1 \over p-1} },
  \end{align*}
  so
  \begin{align*}
    M[u_{n+1}]^{-2 \over p-1} & \geq M[u_n]^{ (2 \gamma - 1) (p-1) - p
                                -1 \over p - 1} = M[u_n]^{-2 \alpha \over p-1},
  \end{align*}
  yielding \eqref{e:Mmonotone}.

\end{proof}

\begin{remark}
  In the case that $\gamma =\gamma_\star = {p \over p-1}$ (and
  $\alpha = 0$), we immediately have that $M[\phi_n] \leq 1$ for all
  $n \geq 1$.  Thus, $\gamma_\star$ plays a distinguished role in both
  the linear and nonlinear analysis of the problem.
\end{remark}

Using this result, we can obtain the following estimates on the entire
sequence.  The main tools are the following uniform estimates for the
sequence of solutions.

\begin{proposition} \label{p.regseq} Assume Assumptions
  \ref{a:domain}, \ref{a:subcrit}, and \ref{a:gamma} hold.  Let $u_n$
  be a sequence generated via \eqref{e:petvia_iter} with nontrivial
  $u_0$.

  \begin{itemize}

  \item There exist positive constants $C_1$ and $C_2$ {depending only
      on $\Omega$, $d$,} $\gamma$, $p$, $\| u_0 \|_{H^1}$, and
    $\| u_0\|_{L^{p+1}}$ such that for all $n \geq 0$,
    \begin{equation} \label{H1bd} 0 < C_1 \leq \LN u_n \RN_{H^1} \leq
      C_2.
    \end{equation}.
     
  \item There exists a positive constant $C_3$ {(with the same
      dependencies as $C_1$ and $C_2$)} such that for all $n\geq 1$,
    {\begin{equation} \label{W2bd} C_3 \geq \begin{cases}
          \LN u_n \RN_{W^{2, {2}}}   & d = 2, \\
          \LN u_n \RN_{W^{2, \kappa}} & d\geq 3, \kappa =
          \min\left\{\frac{2^*}{p}, 2\right\}.
        \end{cases}
      \end{equation}}

  \item If either $\p \Omega$ is smooth or $d = 1$, then there exists
    $k_0\in \mathbb{N}$, depending only on $d$ and $p$, such that for
    all $n\geq n_0$, $u_n \in C^{2(n - n_0)}$.

  \end{itemize}

\end{proposition}

The proof of the proposition relies on the regularity of weak
solutions to the elliptic equation $(I-\Delta)v = f$.  We state these
classical $L^q$ regularity results in the following two lemmas for
future reference.  The first addresses the situation of a smooth
boundary and can be found in, e.g., Theorems~9.13 and 9.15
of~\cite{GilbargTrudinger}.  The second concerns the case of a convex
domain.  It is a slight variant of Corollary~1 of~\cite{fromm:1993};
see also~\cite{adolfsson:1993}.

\begin{lemma} Let $\Omega \subset \mathbb{R}^d$ be an open, bounded
  set, and assume that either $d = 1$ or $\partial\Omega$ is
  smooth. Let $f \in L^q(\Omega)$ for some $1 < q < \infty$. Then the
  equation $(I - \Delta)v = f$ in $\Omega$ with $v = 0$ on
  $\partial \Omega$ has a unique weak solution $v \in W^{2,q}(\Omega)$
  satisfying the estimate
  \begin{equation}\label{smoothReg}
    \| v\|_{W^{2,q}(\Omega)} \leq C \big(\|f\|_{L^q(\Omega)} + \|v\|_{L^{q}(\Omega)}\big).  
  \end{equation}        
\end{lemma}

\begin{lemma}
  Let $f \in L^q(\Omega) \cap H^{-1}(\Omega)$ for some $1 < q \leq 2$,
  and assume that $d \geq 2$ with $\Omega$ satisfying
  Assumption~\ref{a:domain}.  Let $v \in H^1_0(\Omega)$ be the unique
  weak solution to $(I - \Delta)v = f$ in $\Omega$ with $v = 0$ on
  $\partial \Omega$, and further assume that $v \in L^q(\Omega)$.
  Then $v \in W^{2,q}(\Omega)$, and
  \begin{equation}\label{convexReg}
    \| v\|_{W^{2,q}(\Omega)} \leq C \big( \|f\|_{L^q(\Omega)} + \|v\|_{L^q(\Omega)}\big).  
  \end{equation}
\end{lemma}

\begin{proof}[Proof of Proposition~\ref{p.regseq}]

  To establish the regularity of the sequence we use the estimates
  generated from Lemma~\ref{l.ineqs}.  In the following, we focus on
  $d \geq 3$. The cases $d=1,2$ are somewhat simpler.

  \noindent {\bf Step 1.}  We first consider the growth of the
  $H^1$-norms of the $u_n$'s generated by the iteration:
  \begin{equation*}
    \begin{split}
      \LN u_{n+1} \RN_{H^1}
      & \stackrel{\eqref{e:ineq1},\eqref{Sobolev}}{\leq} C_p M [u_n]^\gamma  \LN u_n \RN_{L^{p+1}}^p  =  C_p  \LN u_n \RN_{H^1}^{2 \gamma}    \LN u_n \RN_{L^{p+1}}^{p - (p+1) \gamma} \\
      & \stackrel{\eqref{reverseSobolev},\eqref{Sobolev}}{\leq}
      C_p^{p+1} C_0^{\gamma(p+1)} \LN u_n \RN_{H^1}^{\gamma - \gamma p
        + p}.
    \end{split}
  \end{equation*}
  Thus,
  \begin{equation} \label{alphaupper} \LN u_{n+1} \RN_{H^1} \leq A \LN
    u_n \RN_{H^1}^{\alpha},
  \end{equation}
  where we have taken
  \[
  A = \max\{ C_p^{p+1} C_0^{\gamma (p+1)} , 1\}.
  \]
  Likewise, \eqref{e:ineq2}, \eqref{reverseSobolev}, and
  \eqref{Sobolev} imply the lower bound
  \begin{equation} \label{alphalower} \LN u_{n+1} \RN_{H^1} \geq B \LN
    u_n \RN^\alpha_{H^1},
  \end{equation}
  where we have set
  \[
  B = \min\{ { C_p^{-\gamma ( p+1) } C_0 ^{-p-1}} , 1 \}.
  \]

  \noindent{\bf Step 2.}  Note that \eqref{H1bd} holds trivially when
  $\alpha = 0$.  Now consider $0 < \alpha < 1$. By iterating, and
  using that $A \geq 1$,
  \begin{equation*}
    \LN u_{n} \RN_{H^1} \leq A^{\sum_{k=0}^{n-1} \alpha^k} \LN u_0 \RN_{H^1}^{\alpha^{n}} \leq A^{1 \over 1 - \alpha } \LN u_0 \RN_{H^1}^{\alpha^{n}}.
  \end{equation*}
  Similarly, since $B \leq 1$ we have
  \begin{equation*}
    \LN u_{n} \RN_{H^1}\geq B^{\sum_{j=0}^{n-1} \alpha^j} \LN u_0 \RN_{H^1}^{\alpha^{n}}  \geq B^{1 \over 1 - \alpha} \LN u_0 \RN_{H^1}^{\alpha^{n}}.
  \end{equation*}

  Next, consider $-1 < \alpha < 0$.  Using both \eqref{alphaupper} and
  \eqref{alphalower}, we find that
  \begin{equation*}
    \LC B\over A^{|\alpha|} \RC \LN u_{n-2} \RN_{H^1}^{|\alpha|^2} \leq  \LN u_{n} \RN_{H^1} \leq \LC A \over B^{|\alpha|} \RC \LN u_{n-2}\RN_{H^1}^{|\alpha|^2}.
  \end{equation*}
  Again, by iteration,
  \begin{equation*}
    \LN u_{n} \RN_{H^1}  \leq 
    \begin{cases}
      \LC {A \over B^{|\alpha|} } \RC^{ 1
        \over 1 - |\alpha|^2} \LN u_0  \RN_{H^1}^{ |\alpha|^n} & \text{$n$ even},   \\
      \LC {A \over B^{|\alpha|} } \RC^{ 1 \over 1 - |\alpha|^2}
      A^{|\alpha|^{n-1}} \LN u_0 \RN_{H^1}^{- |\alpha|^n} &\text{$n$
        odd},
    \end{cases}
  \end{equation*}
  and
  \begin{equation*}
    \LN u_{n} \RN_{H^1}  \geq 
    \begin{cases}
      \LC {B \over A^{|\alpha|} } \RC^{ 1 \over 1 - |\alpha|^2} \LN
      u_0 \RN_{H^1}^{ |\alpha|^n} & \text{$n$ even}, \\ \LC {B \over
        A^{|\alpha|} } \RC^{ 1 \over 1 - |\alpha|^2}
      A^{|\alpha|^{n-1}} \LN u_0\RN_{H^1}^{- |\alpha|^n} & \text{$n$
        odd}.
    \end{cases}
  \end{equation*}
  These estimates imply \eqref{H1bd}.

  \noindent{\bf Step 3.}  We now prove \eqref{W2bd} for $d \geq 3$.
  We recall that $\kappa = \min\left\{2, \frac{2^*}{p}\right\}$ and
  consider the two cases, $\kappa = 2$ or $\kappa = \frac{2^*}{p}$.
  First suppose that $\kappa = \frac{2^*}{p}$ so
  $\frac{2^*}{p} \leq 2$.
	
  Using the regularity result~\eqref{convexReg} with $v = u_{n+1}$,
  $f = M[u_n]^\gamma |u_n|^{p - 1} u_n$, and $q = {2^* \over p}$, we
  obtain
  \begin{equation*}
    \begin{split}
      \LN u_{n+1} \RN_{W^{2, {2^* \over p}}} \leq~& C M
      [u_n]^\gamma \LN \abs{u_{n}}^{p-1}u_n \RN_{L^{2^*\over p}} + C \|u_{n+1}\|_{L^{2^*\over p}} \\
      =~& C { \LN u_n \RN_{H^1}^{2 \gamma} \over \LN u_n
        \RN_{L^{p+1}}^{\gamma(p+1)}} \LN u_n \RN_{L^{2^*}}^p + C C_p\|u_{n+1}\|_{H^1} \\
      \leq~& \big( C C_{2^*}^p C_0^{\gamma(p+1)} + ACC_p\big) \LN u_n
      \RN_{H^1}^{\alpha},
    \end{split}
  \end{equation*}
  where we have used the reverse Sobolev
  embedding~\eqref{reverseSobolev}, standard Sobolev embedding,
  and~\eqref{alphaupper} in the final inequality.  As
  $\LN u_n \RN_{H^1} \leq C_2$ uniformly in $n$, the
  $W^{2,{2^*\over p}}$ bound \eqref{W2bd} follows.
	
  Now assume that $\kappa = 2$ so that $2 \leq {2^* \over p}$.  We may
  again invoke~\eqref{convexReg} with $q = 2$ to deduce
  \begin{align*}
    \|u_{n+1}\|_{W^{2,2}} \leq~& C M
                                 [u_n]^\gamma \LN \abs{u_{n}}^{p-1}u_n\RN_{L^2} + C \|u_{n+1}\|_{L^{2}} \\
    \leq~& C M
           [u_n]^\gamma \LN\abs{u_{n}}^{p-1}u_n \RN_{L^{2^*\over p}} + C \|u_{n+1}\|_{H^1}.
  \end{align*}
  Proceeding as in the previous case gives the desired result.
		
  A similar argument (without needing separate cases) holds for $d=2$.

  \noindent{\bf Step 4.} When the boundary is smooth, we can further
  iterate the algorithm and generate improved regularity.  But first
  we prove two small claims to organize the argument.  If $v \in L^s$
  and $f \in L^q$ with $1 < s < q < \infty$ with $s < {d \over 2}$ and
  $(I - \Delta ) v = f$ then
  \begin{equation} \label{e:w2estint2} \LN v \RN_{L^{sd \over d - 2
        s}} \leq C \LC \LN v \RN_{L^s} + \LN f \RN_{L^q} \RC.
  \end{equation}

  To show this we note that since the domain is compact,
  $\LN f \RN_{L^s} \leq C \LN f \RN_{L^q}$ with a constant $C$
  depending on $\Omega$.  {Since the boundary is smooth}, by
  \eqref{smoothReg} we have
  \[
  \LN v \RN_{W^{2,s}} \leq C \LC \LN v \RN_{L^s} + \LN f \RN_{L^q}
  \RC,
  \]
  and Sobolev embedding implies \eqref{e:w2estint2}.
 
  \noindent{\bf Step 5.} Next, we claim that if $v \in L^r$ and
  $f \in L^q$ with $1 < r < q < \infty$ and $(I - \Delta ) v = f$ then
  either $v \in C^{0,\beta}$ for some $\beta \in (0,1]$ or
  $v \in {W^{2,q}}$ with
  \begin{equation} \label{e:w2estint3} \LN v \RN_{W^{2,q}} \leq C \LC
    \LN v \RN_{L^r} + \LN f \RN_{L^q} \RC.
  \end{equation}
  
  We can prove \eqref{e:w2estint3} by iterating.  Set $r_0 = r$, {and
    consider a sequence of iterates $r_j$ such that $v \in W^{2,r_j}$
    and defined as follows.}  If $r_j > {d \over 2}$, then
  $v \in C^{0,\beta}$ for some $0 < \beta \leq 1$ by Morrey's
  inequality, and the iteration terminates.  On the other hand, if
  $r_j = {d \over 2}$, then $v \in L^s$ for all $s < \infty$.  In
  particular, $v \in L^q$, so {\eqref{smoothReg} implies
    $v \in {W^{2,q}}$, and the iteration stops.  Finally, if
    $r_j < {d \over 2}$, then by Step 4,
    $v \in L^{r_j d \over d - 2 r_j}$.  If
    ${r_j d \over d - 2 r_j} \geq q$, then $v \in L^q$,
    \eqref{smoothReg} implies $v \in {W^{2,q}}$, and the iteration
    halts.  If ${r_j d \over d - 2 r_j} < q$, then set
    $r_{j+1} = {r_j d \over d - 2 r_j}$ and continue the iteration.}

  \noindent{\bf Step 6.}
  
  We claim that there exists $n_0 \in \mathbb{N}$ such that
  $u_{n_0} \in C^{0,\beta}$ for a $0 < \beta \leq 1$.  Note that if
  $u_n \in W^{2,r_n}$ with $r_n > {d \over 2}$, then by Morrey's
  inequality, $u_n \in C^{0,\beta_n}$ for
  $\beta_n = 2 - {d \over r_n }$; therefore, we will iterate the
  Petviashvilli algorithm until we reach an $n$ such that
  $r_n > {d \over 2}$.  In the following we will use \eqref{H1bd}
  which implies
  \begin{equation} \label{e:unibd2s} \LN u_{j} \RN_{L^{2^*}}\leq C
    \qquad \hbox{ for all } j \in \mathbb{N}.
  \end{equation}

  If $r_{n} < {d \over 2}$, then we claim that either
  $u_{n+1} \in C^{0,\beta}$ for some $0 < \beta \leq 1$ or
  $u_{n+1} \in W^{2, r_{n+1}}$ with
  \begin{equation} \label{e:recursion} r_{n+1} = {d r_n \over p (d - 2
      r_n)}.
  \end{equation}   
  First note that by Sobolev embedding $u_n \in L^{s_n}$ with
  ${1 \over s_n} = {1\over r_n} - {2 \over d}$ so
  $|u_n|^{p-1} u_n \in L^{s_n \over p}$.  Therefore, if
  ${s_n \over p} \leq 2^*$, then \eqref{e:unibd2s} implies
  $u_{n+1} \in L^{s_n \over p}$, and we can use \eqref{smoothReg} with
  $q = {s_n \over p}$ to see that $u_{n+1} \in W^{2, r_{n+1}}$.  If
  ${s_n \over p} > 2^*$, then we use Step 5, which implies either
  $u_{n+1} \in C^{0,\beta}$ for some $0 < \beta \leq 1$ or
  $u_{n+1} \in W^{2,{r_{k+1}}}$ in which case $r_{k+1}$ satisfies
  \eqref{e:recursion}.  We note that \eqref{e:recursion} is an
  increasing sequence so long as $p < d^\dagger$.  This iteration can
  be solved explicitly with $r_1 = {2^* \over p}$:
  \[
  r_n = { 2 d ( p -1) \over 4 p - p^n \LC (d+2) - (d-2) p \RC}.
  \]
  We now iterate the sequence until $r_{n_0} > {d\over 2}$ at which
  point $W^{2,r_{n_0}}$ embeds in a H\"older space.
  
  If $r_n = {d \over2 }$, then $u_n \in L^s$ for all $ s < \infty$.
  We take $s = {2d \over p}$ so $ |u_n|^{p-1}u_n \in L^{2 d \over p}$.
  Following a similar argument as above shows that
  $u_{n+1} \in W^{2, {2 d \over p}} \subset C^{0,\beta}$ for some
  $0 < \beta\leq1$.

  Finally, we can use classical Schauder estimates to show that
  $u_{n} \in C^{2 (n - n_0), {\beta}}$ for all $n \geq n_0$.

\end{proof}


\subsection{Global Convergence}

\begin{proof}[Theorem \ref{t:global}]

  First, we recall Lemma~\ref{l.ineqs} which implies that if
  $\gamma = {p \over p-1}$ then $\alpha = 0$ and $M[u_n] \leq 1$ for
  all $n \geq1$.

  \noindent{\bf Step 1.} We will first show that
  \begin{equation} \label{e:Mto1} \lim_{n\to \infty} M[u_n]= 1.
  \end{equation}
  By \eqref{e:Mgrowth}, we have
  $M[u_n]^{\gamma-1} \leq { \| u_{n+1} \|_{H^1} / \| u_{n} \|_{H^1} }$
  so that
  \begin{align*}
    M[u_n]^{\gamma-1} M[u_{n-1}]^{\gamma-1} \cdots M[u_0]^{\gamma-1}
    \leq { \| u_{n+1} \|_{H^1} \over \| u_{n} \|_{H^1} } { \| u_{n}
    \|_{H^1} \over \| u_{n-1} \|_{H^1} } \cdots { \| u_{1} \|_{H^1}
    \over \| u_{0} \|_{H^1} }.
  \end{align*}
  Rewriting this as
  \begin{align*}
    \LC \prod_{k=0}^n M[u_k] \RC^{\gamma -1 } \leq { \| u_{n+1}
    \|_{H^1} \over \| u_{0} \|_{H^1} } ,
  \end{align*}
  we can, from \eqref{H1bd}, conclude that
  \begin{equation}
    \prod_{k=0}^n M[u_k] \leq  \LC { \| u_{n+1} \|_{H^1} \over  \| u_{0} \|_{H^1} } \RC^{1 \over \gamma -1 } \leq C.
  \end{equation}

  We repeat this argument with the lower bound \eqref{e:Mgrowth2}.  In
  particular we get
  \begin{align*}
    M[u_n]^{\gamma-1}  \cdots M[u_0]^{\gamma-1} &\geq { \| u_{n+1} \|^2_{H^1} \over  \| u_{n} \|^2_{H^1} }  \cdots { \| u_{1} \|^2_{H^1} \over  \| u_{0} \|^2_{H^1} } { \| u_{n} \|_{L^{p+1}} \over  \| u_{n+1} \|_{L^{p+1}} } \cdots { \| u_{0} \|_{L^{p+1}} \over  \| u_{1} \|_{L^{p+1}} }  \\
                                                &\quad = { \| u_{n+1} \|^2_{H^1} \over \| u_{0} \|^2_{H^1} } { \| u_{0}
                                                  \|_{L^{p+1}} \over \| u_{n+1} \|_{L^{p+1}} }.
  \end{align*}
  Using \eqref{H1bd} and \eqref{e:revSob}, we have
  \begin{equation} \label{e:prodlow} \prod_{k=0}^n M[u_k] \geq \LC {
      \| u_{n+1} \|^2_{H^1} \over \| u_{0} \|^2_{H^1} } { \| u_{0}
      \|_{L^{p+1}} \over \| u_{n+1} \|_{L^{p+1}}} \RC^{1 \over \gamma
      -1 } \geq C > 0.
  \end{equation}

  Now consider the series $\sum_{k=1}^\infty \log M[u_k]$.  Since
  $\log M[u_k] \leq 0$ by \eqref{e:Mmonotone} and since
  \[
  \sum_{k=0}^\infty \log M[u_k] \geq {\log C} > -\infty
  \]
  from \eqref{e:prodlow}, the series converges.  Therefore,
  $\lim_{n \to \infty} \log M[u_n] = 0$.  This implies \eqref{e:Mto1}.

  \noindent{\bf Step 2.} Next we show that
  $\| u_{n+1} - u_n \|_{H^1} = \littleo_n(1)$ for all $n \geq1$.  We
  begin with the estimate
  \begin{align*}
    M[u_n]^{\gamma-1} \stackrel{\eqref{e:Mgrowth2}}{\geq} { \LN
    u_{n+1} \RN^2_{H^1} \over \LN u_{n} \RN^2_{H^1}} { \LN u_{n}
    \RN_{L^{p+1}} \over \LN u_{n+1} \RN_{L^{p+1}} }
    \stackrel{\eqref{e:MLp}}{\geq} { \LN u_{n+1} \RN^2_{H^1} \over \LN
    u_{n} \RN^2_{H^1}} \LC { M[u_{n+1}] \over M[u_n] } \RC^{1 \over
    p-1}.
  \end{align*}
  From this we obtain
  \begin{equation} \label{e:phiH1grow} \LN u_{n+1} \RN^2_{H^1} \leq
    M[u_n]^{\gamma-1} \LC { M[u_{n+1}] \over M[u_n] } \RC^{1 \over
      p-1} \LN u_{n} \RN^2_{H^1}.
  \end{equation}
  Therefore,
  \begin{align*}
    \LN u_{n+1} - u_n \RN_{H^1}^2
    & = \LN u_{n+1} \RN_{H^1}^2 + \LN u_{n} \RN_{H^1}^2 - 2 \LC u_{n+1} , u_n \RC \\
    & \stackrel{\eqref{e:discident}}{=}  \LN u_{n+1} \RN_{H^1}^2 + \LN u_{n} \RN_{H^1}^2 - 2 M[u_n]^{\gamma - 1} \LN u_{n} \RN_{H^1}^2  \\
    & \stackrel{\eqref{e:phiH1grow}}{\leq} \LN u_{n} \RN_{H^1}^2\LB 1
      - 2 M[u_n]^{\gamma - 1} + M[u_n]^{\gamma-1}
      \LC { M[u_{n+1}] \over M[u_n] }  \RC^{1 \over p-1} \RB  \\
    & \stackrel{\eqref{H1bd}}{\leq} C \LB 1 - 2 M[u_n]^{\gamma - 1} +
      M[u_n]^{\gamma-1} \LC { M[u_{n+1}] \over M[u_n] } \RC^{1 \over
      p-1} \RB.
  \end{align*}
  Since $M[u_n] \to 1$, the right-hand side vanishes as
  $n \to \infty$.

  \noindent {\bf Step 3.}  We claim that our iteration converges to a
  strong solution to \eqref{e:boundstate}.  Again, we present the
  argument for $d \geq 2$. By Proposition \ref{p.regseq}, we know that
  our sequence $\{ u_n \}$ is uniformly bounded in $W^{2,\kappa}$
  (where $\kappa$ is defined in Theorem~\ref{t:global}) so we can
  extract a weakly convergent subsequence, $\{u_{n_k}\}$, in
  $W^{2,\kappa}$, with limit $\phi \in W^{2,\kappa}$.  By Rellich,
  $W^{2,\kappa}$ is compactly embedded in $H^1$ so this subsequence
  will converge strongly to $\phi$ in $H^1_0$.

  From Step 2, we also know that $u_{n_k +1}\to \phi$, strongly in
  $H^1_0$.  For any $\psi \in H^1_0$,
  \[
  \inner{u_{n_k +1}}{\psi}_{H^1} =M[u_{n_k}]^\gamma
  \inner{\abs{u_{n_k}}^{p-1}u_{n_k}}{\psi},
  \]
  and from the strong $H^1$ convergence of $\left\{u_{n_k +1}\right\}$
  and Step 1, we have
  \begin{align*}
    \inner{u_{n_k +1}}{\psi}_{H^1} &\to \inner{\phi}{\psi}_{H^1}, \\
    M[u_{n_k}]^\gamma  \inner{\abs{u_{n_k}}^{p-1}u_{n_k}}{\psi}  &\to \inner{\abs{\psi}^{p-1}\psi}{\psi}.
  \end{align*}
  Thus, the limit $\phi$ is a weak solution of \eqref{e:boundstate}.
  Finally, by \eqref{W2bd} we get the improved regularity of $\phi$.

  \noindent{\bf Step 4.} When $\p \Omega$ is smooth, we claim that
  $u_n \to {\phi}$ in $C^k$ and that $\phi$ is a $C^\infty$ solution
  to \eqref{e:boundstate}.  Since $u_n \in C^{k,{\beta}}$ for, say,
  $k\geq 3$, by {the compact embedding of $C^{k,{\beta}}$ in $C^k$},
  there exists a subsequence $\{ u_{n_k} \} \subset \{ u_n \}$ and a
  limit ${\phi} \in C^k$ such that $u_{n_k} \to {\phi}$ strongly in
  $C^k$.

  Now, from the iteration scheme,
  \[
  u_{n_k+1} = (I - \Delta)^{-1} \LC M [u_{n_k}]^\gamma |u_{n_k}|^{p-1}
  u_{n_k} \RC
  \]
  is a strongly convergent sequence in $C^{k+2}$.  Therefore,
  $\{ u_{n_k+1} \}$ is a Cauchy sequence in $C^k$ with a limit
  ${\overline{\phi}}$.  However, by Step 2, as $n_k$ tends to
  infinity,
  \[
  \LN {\phi} - {\overline{\phi}} \RN_{H^1} \leq \norm{\phi -
    u_{n_k}}_{H^1} + \norm{u_{n_k} - u_{n_k+1}}_{H^1} +
  \norm{u_{n_k+1} - {\overline{\phi}}}_{H^1}
  \]
  goes to zero.  Therefore, ${\overline{\phi}} = {\phi}$.  Returning
  to the scheme, we find $u_{n_k+1} \to {\phi}$ and
  \[
  (I - \Delta ) {\phi} = |{\phi}|^{p-1} {\phi}
  \]
  since $M[u_n] \to 1$ for all $n$.

\end{proof}

\begin{remark}
  It is natural to ask whether the sequence generated by the
  Petviashvilli iteration is, in fact, Cauchy; however, we are unable
  to show this, unless the limit $\phi$ satisfies the assumptions of
  Theorem~\ref{t:local} (see Corollary~\ref{c:combo}).  On the other
  hand, one can follow Step 3 of the proof of Theorem~\ref{t:global}
  and use induction to show for any fixed $\ell \in \mathbb{Z}^+$ that
  $u_{n_k + \ell}$ is a subsequence converging to ${\phi}$.

  In order to prove that a.e. sequence is Cauchy, it would be
  sufficient to show that
  $| 1 - M[u_n]^{\gamma-1} | \leq C n ^{-\sigma}$ for some
  $\sigma >1$.

\end{remark}

\begin{proof}[Corollary \ref{c:combo}]
  If the subsequential limit, $\phi$, satisfies the spectral
  condition, then Theorem \ref{t:local} applies to it.  Some element
  of the subsequence will be in the local basin of attraction of
  $\phi$, and the linear theory implies the whole sequence converges
  to $\phi$.
\end{proof}


\subsection{Tendency to the Groundstate}
\label{ss.ground}
As our numerical experiments show, \eqref{e:petvia_iter} appears to
always produce a ground state solution.  Here, we provide some insight
into this behavior.  We first define a modified energy,
\begin{equation*} \label{e:modE}
  E_M[\psi] \equiv \begin{cases}  {1\over 2} \LN \psi \RN^2_{H^1} - {1\over p+1} M[\psi] \LN \psi \RN^{p+1}_{L^{p+1}}  & M[\psi] \leq 1, \\
    + \infty & M[\psi] > 1.
  \end{cases}
\end{equation*}
In the case that $M[\psi] =1$, we recover the standard energy
$E_1 [\psi] := E[\psi]$, as defined by \eqref{e:egroundstate}.

We can rewrite $E_M[\psi]$ as
\begin{equation}
  \label{e:modE2} 
  E_M[\psi] = \LN \psi \RN^2_{H^1}\LB {1\over 2} - {1\over p+1}\RB .
\end{equation}

\begin{theorem}
  \label{t:minimizer}
  Let $u_n$ be an iteration generated by \eqref{e:petvia_iter} with
  $u_0$ nontrivial initial data and $\gamma = {p \over p-1}$.  Then
  \begin{equation}
    \label{e:Liapunov1} 
    \liminf_{n \to \infty} E_1[u_n]\leq E_M[u_0].
  \end{equation}

  Given the solution $\phi$ extracted from Theorem~\ref{t:global}, we
  possess
  \begin{equation}
    \label{e:Liapunov2} 
    E_1[\phi] \leq E_M[u_0].
  \end{equation}
\end{theorem}

The proof of Theorem~\ref{t:minimizer} follows from
Lemma~\ref{l:Eminimizer} below.

\begin{lemma}
  \label{l:Eminimizer}
  
  Assume that $\gamma=\gamma_\star=p/(p-1)$ and that
  \begin{equation}
    0 < M[u_k] \leq M[u_{n} ] \leq 1
  \end{equation}
  for some $0 \leq k <n$. Then
  \begin{equation*}
    \label{e:Emonotone} 
    E_M[u_{n}] \leq E_M [u_k],
  \end{equation*}
  with the inequality strict if $M[u_{n}] < M[u_k]$.

  In particular, if $0 < M[u_k] \leq 1$, then
  \begin{equation*} \label{e:Emonotone2} \liminf_{n\to\infty} E_1
    [u_n]\leq E_M [u_k].
  \end{equation*}
\end{lemma}

\begin{proof}
  We first obtain the estimate,
  \begin{equation*}
    {M[u_{n+1}] \over M[u_n] } = { \LN u_{n+1} \RN^2_{H^1} \over \LN
      u_{n} \RN^2_{H^1} } {\LN u_{n} \RN^{p+1}_{L^{p+1}} \over \LN
      u_{n+1} \RN^{p+1}_{L^{p+1}} } \leq { \LN u_{n} \RN^{p-1}_{H^1}
      \over \LN u_{n+1} \RN^{p-1}_{H^1} }\Rightarrow  { \LN u_{n+1} \RN^{2}_{H^1} \over \LN u_{n} \RN^{2}_{H^1} } \leq \LC
    {M[u_{n}] \over M[u_{n+1}] }\RC^{2 \over p-1}.
  \end{equation*}

  We now consider the telescoping product and use \eqref{e:modE2}:
  \begin{equation*}
    \begin{split}
      & {E_M[u_{n+1}] \over E_M[u_{n}]} {E_M[u_{n}] \over
        E_M[u_{n-1}]} \cdots {E_M[u_{k+1}] \over E_M[u_{k}]}  \\
      &\leq \LC {M[u_{n}] \over M[u_{n+1}] }\RC^{2 \over p-1} \LC
      {M[u_{n-1}] \over M[u_{n}] }\RC^{2 \over p-1} \cdots \LC
      {M[u_{k}]\over M[u_{k+1}] }\RC^{2 \over p-1} .
    \end{split}
  \end{equation*}
  Therefore,
  \begin{align*}
    {E_M[u_{n+1}] \over E_M[u_{k}]} \leq \LC {M[u_{k}] \over
    M[u_{n+1}] }\RC^{2 \over p-1} .
  \end{align*}
  Letting $n \to \infty$ yields
  \begin{align*}
    \liminf_{n\to\infty} { E_1[{u_n}] \over E_M[u_{k}]} \leq \LC {M[u_{k}] }\RC^{2 \over p-1} \leq 1.
  \end{align*}

\end{proof}

We conclude with a proof of the theorem.
\begin{proof}[Proof of Theorem~\ref{t:minimizer}]
  The proof follows from noting that \eqref{e:Liapunov1} holds when
  $M[u_0] \leq1$, and \eqref{e:Liapunov1} holds trivially if
  $M[u_0] \geq 1$.  Inequality \eqref{e:Liapunov2} follows by looking
  at the iteration in Lemma~\ref{l:Eminimizer} with sequence elements
  $u_{n_k}$ from Theorem~\ref{t:global}.
\end{proof}

\section{Examples}
\label{s:examples}

In this section, we present a few numerical examples illustrating the
algorithm.

\subsection{1D Examples}

For a first example, we consider the following instance of
\eqref{e:boundstate}
\begin{equation}
  \label{e:boundstate_1d}
  \phi''-\phi + \abs{\phi}^2\phi = 0, \quad \phi(\pm \Xmax) =0,
\end{equation}
for some $\Xmax >0$.  The advantage of considering such a problem is
that its solution can be expressed in terms of a Jacobi
elliptic\footnote{We use the definition $\cn = \cn(x;m)$ rather than
  $\cn = \cn(x; k^2)$.}  function as,
\begin{equation*}
  \label{e:boundstate_1dsoln}
  \phi(x) = (1+\beta^2) \cn\paren{\beta x; \tfrac{1}{2}(1 + \beta^{-2})},
\end{equation*}
where $\beta>0$ solves
$ \cn\paren{\beta \Xmax; \tfrac{1}{2}(1 + \beta^{-2})} = 0$.

\subsection{Profiles}
We discretize \eqref{e:boundstate_1d} with piecewise linear finite
elements on $N+1$ uniformly spaced mesh points in $(-\Xmax, \Xmax)$,
and we approximate the nonlinear term as
\begin{equation}
  \abs{\sum u_j \varphi_j}^{p-1} \paren{\sum u_j \varphi_j} \approx
  \sum \abs{u_j}^{p-1} u_j \varphi_j.
\end{equation}

In this problem, we take $\gamma = \gamma_\star = p/(p-1)$.
Throughout, $\Xmax =2$.  As we seek a positive solution of
\eqref{e:boundstate_1d}, as an initial guess, we take
\begin{equation}
  \label{e:boundstate_1d_guess}
  u_0 = (\Xmax - x)(\Xmax+x).
\end{equation}

The profiles, obtained {via the Petviashvilli iteration} with
different values of $N$, are plotted in Figure \ref{f:1dprofiles}.
All are in visual agreement with the exact solution.  Turning to
convergence, in Figure \ref{f:1dconvergence}, we examine the $H^1$
error and the distance of $M[u_n]$ to 1.  Several desirable properties
are seen.  First, as our analysis has been of the semi-discrete
algorithm (without having introduced a mesh), we would hope to have
mesh insensitivity; the convergence of $M[u_n]$ is relatively
indifferent to $N$.  Next, we see in the $H^1$ convergence figure that
the error reaches a plateau.  This is to be expected, as we must
contend with both discretization and algorithmic error.  We believe
these plateaus reflect the error between the analytic solution and the
discretized one, and the reduction of error with mesh refinement is
consistent with this.

Finally, convergence Figure~\ref{f:1dconvergence} indicates that,
though we let the algorithm run for 100 iterations, we could have
terminated after far fewer, perhaps 10-20 iterations, and obtained a
high quality solution.

\begin{figure}
  \begin{center}
    \includegraphics[width=7cm]{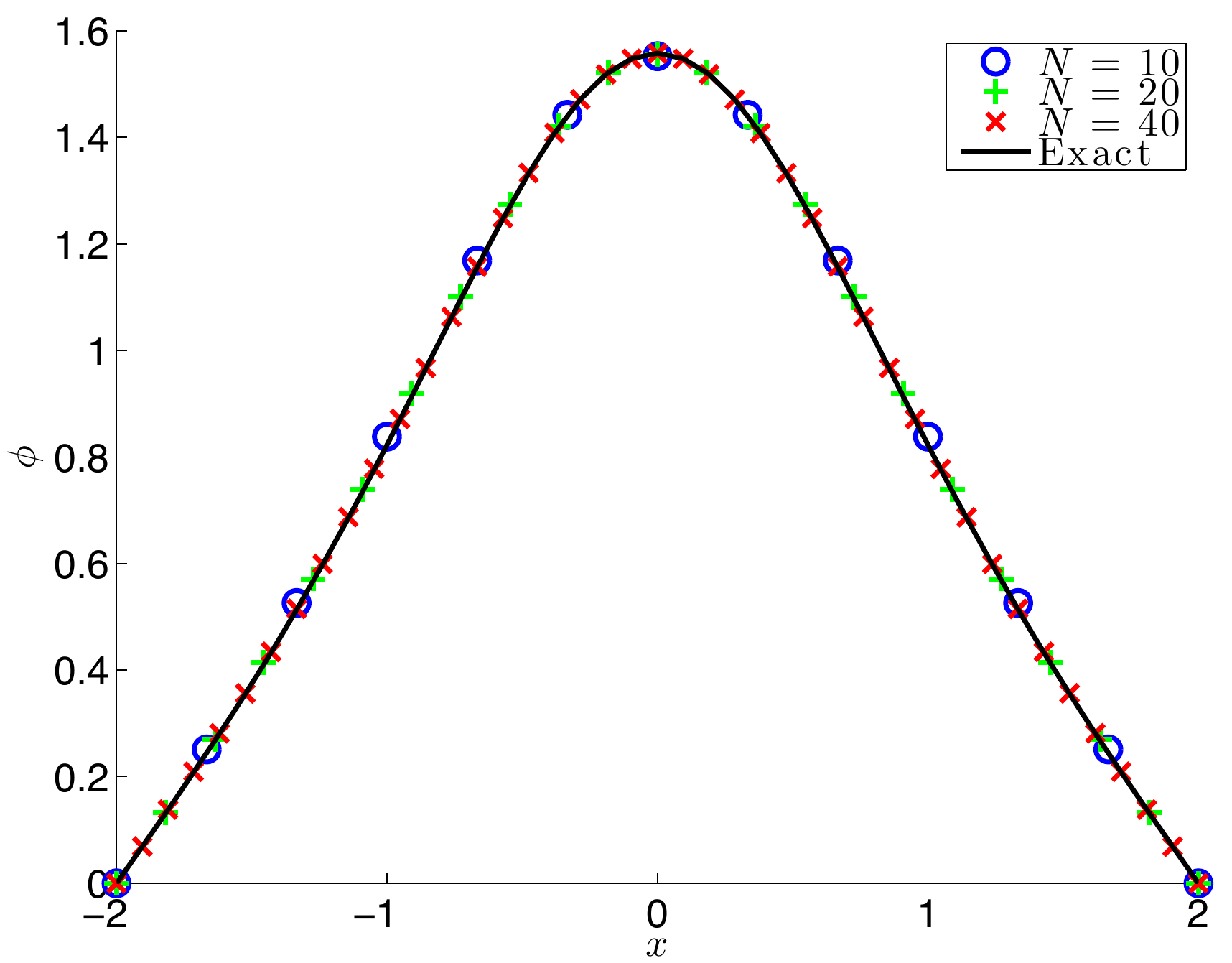}
  \end{center}
  \caption{Solutions to \eqref{e:boundstate_1d} with $\Xmax =2$ at
    different resolutions.  These are the profiles obtained after 100
    iterations of the algorithm.}
  \label{f:1dprofiles}
\end{figure}

\begin{figure}
  \subfigure[]{\includegraphics[width=6.35cm]{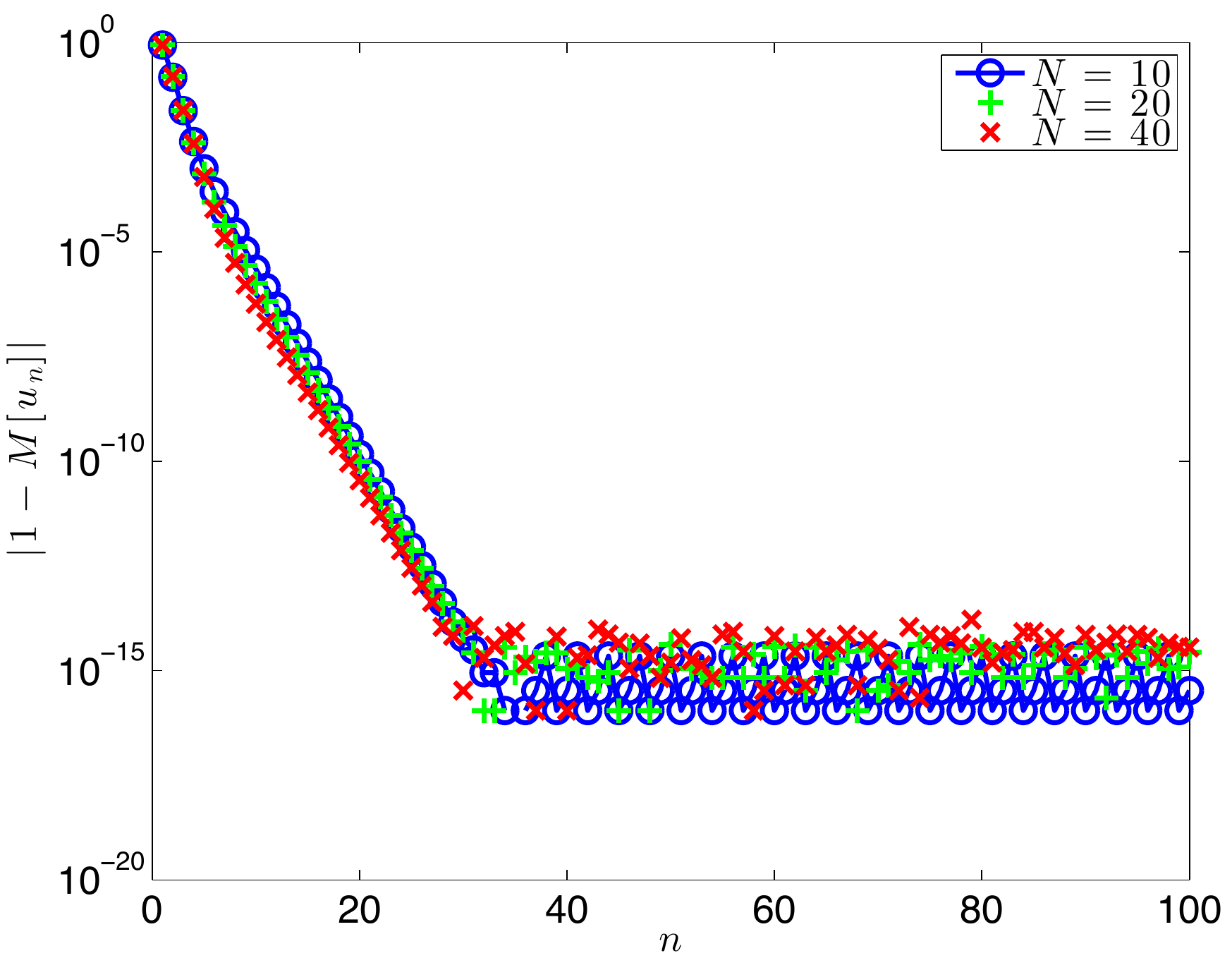}}
  \subfigure[]{\includegraphics[width=6.35cm]{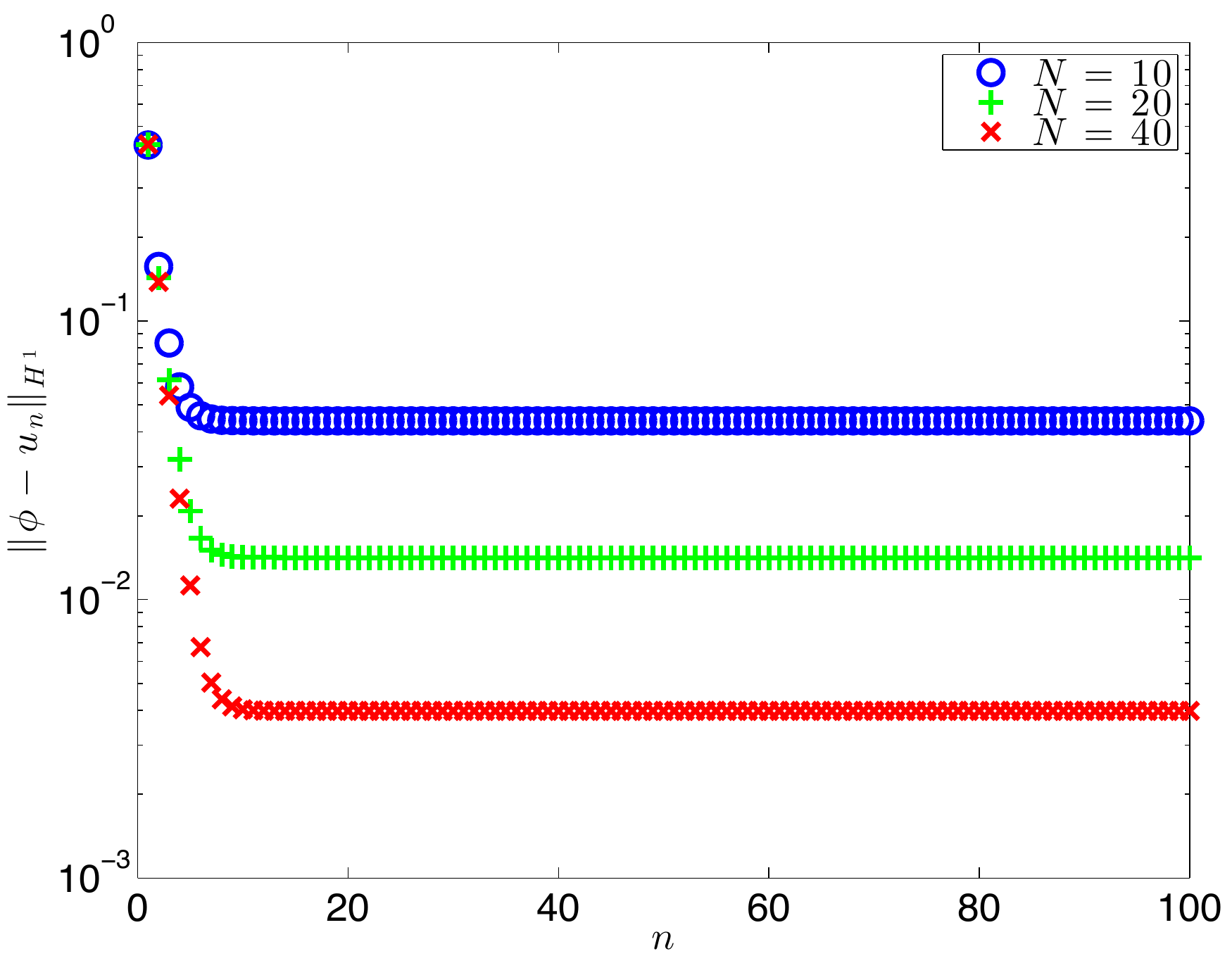}}
  \caption{Convergence of the algorithm for \eqref{e:boundstate_1d},
    measured in terms of $|M[u_n]-1|$ and the $H^1$ error against the
    exact solution.}
  \label{f:1dconvergence}
\end{figure}

\subsection{Spectral Analysis and Linear Convergence}

Key to our linearized analysis {was the spectral assumption that $p$
  was the only eigenvalue of $\calL^{-1}(p\abs{\phi}^{p-1})$ that was
  outside the unit disk}, Assumption \ref{a:spectral}.  In addition,
Proposition \ref{p:lineariter} tells us that, in the case
$\gamma =\gamma_\star$, the contraction mapping constant, $\theta$,
will be the largest eigenvalue smaller than $p$.

Using the solution we have computed, we numerically solve the
equivalent generalized eigenvalue problem
\begin{equation}
  \label{e:calLeig}
  p\abs{\phi}^{p-1} \psi = \nu \calL \psi, 
\end{equation}
and find that this is indeed the case.  The numerically computed
spectrum for different values of $N$ are plotted in Figure
\ref{f:spectrum1d}.  As hoped for, the only eigenvalue in excess of
one is $p=3$, and the spectrum is nonnegative.

\begin{figure}
  \subfigure[]{\includegraphics[width=6.35cm]{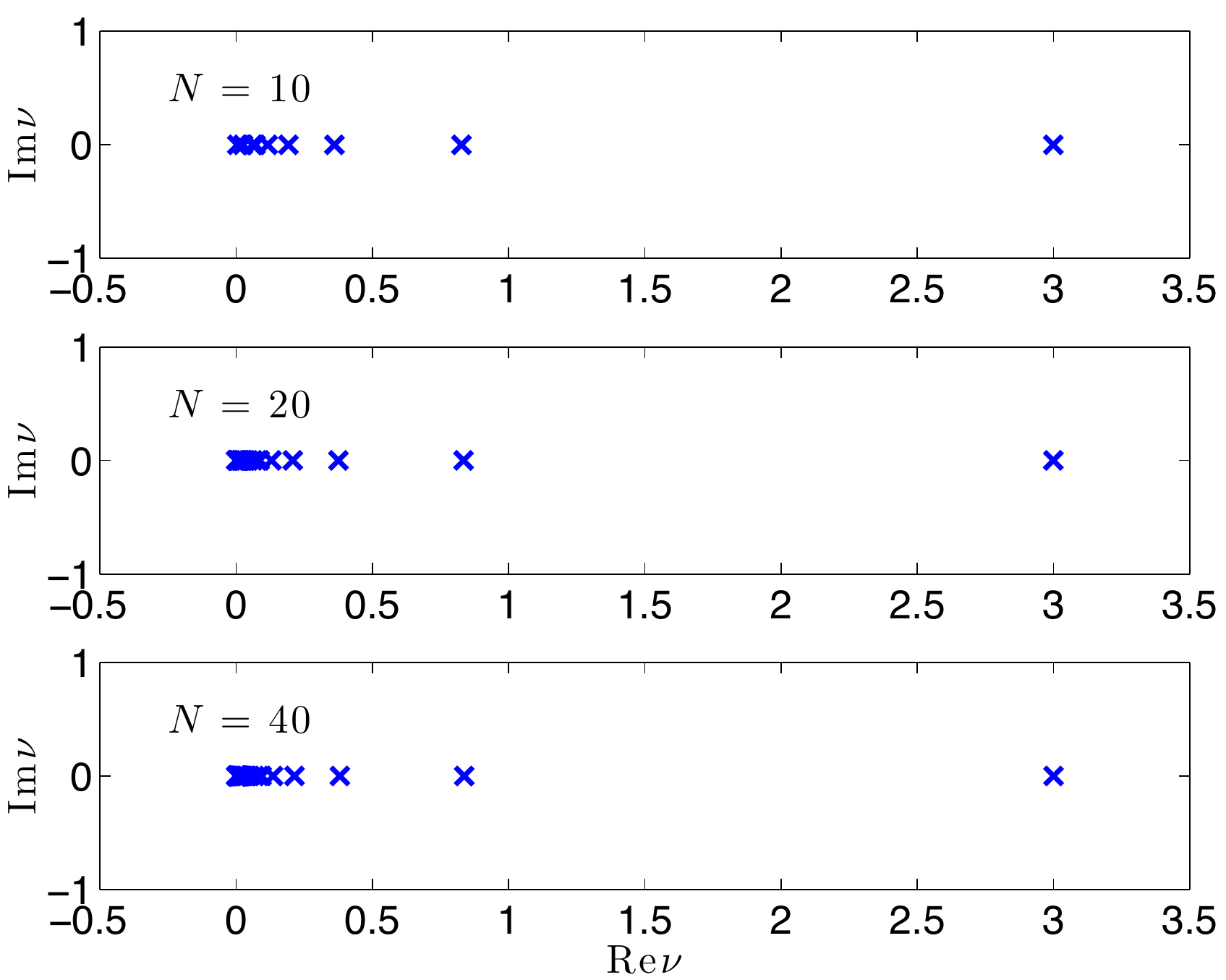} \label{f:spectrum1d}}
  \subfigure[]{\includegraphics[width=6.35cm]{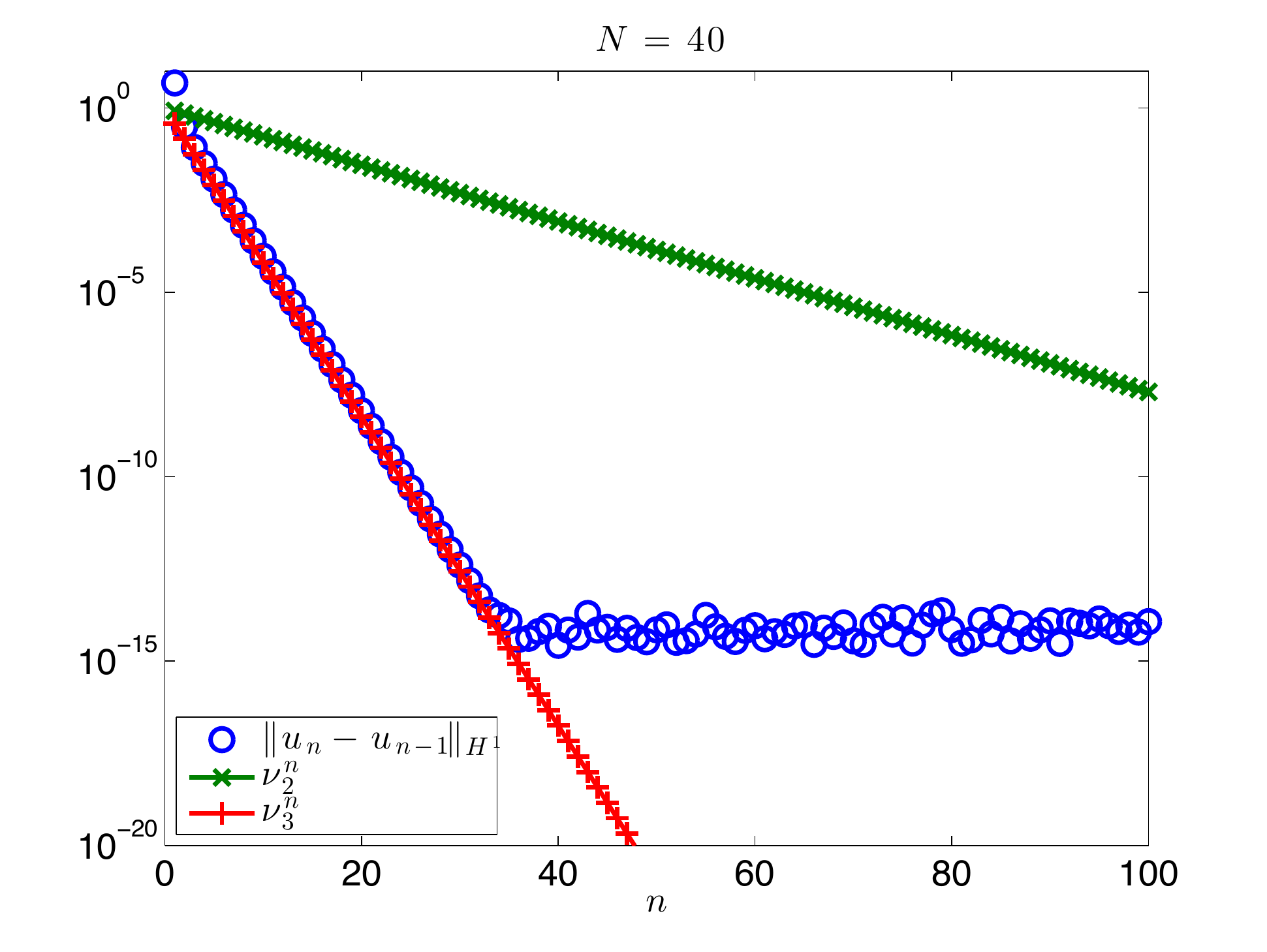} \label{f:analysis1d}}

  \caption{On the left, we have the spectrum of the eigenvalue problem
    \eqref{e:calLeig} with different numbers of mesh points.  In all
    cases, there is a single eigenvalue at $p=3$, and the next largest
    eigenvalue is near $.84<1$.  On the right, we have successive
    differences, measured in the $H^1$ norm, of the iteration scheme.
    Note that while $\mu_\star$ is $\nu_2\approx .84$ from Figure
    \ref{f:spectrum1d}, the convergence rate follows $\nu_3^n$, with
    $\nu_3\approx.38$, the next largest eigenvalue.}

\end{figure}

This begs the question as to whether or not, when $u_n$ is
sufficiently close to $\phi$, the algorithm truly obeys
\eqref{e:linear_conv}.  In Figure \ref{f:analysis1d}, we plot
$\norm{u_{n+1} - u_n}_{H^1}$, which, if we are sufficiently close to
$\phi$, should satisfy the contraction
\begin{equation}
  \norm{u_{n+1} - u_n}_{H^1} \leq \mu_\star\norm{u_{n} - u_{n-1}}_{H^1}.
\end{equation}
Recall that $\mu_\star$ is the largest eigenvalue beneath $p$
associated with \eqref{e:calLeig}.  We plot this difference in Figure
\ref{f:analysis1d}.  Note $\mu_\star= \nu_2\approx .84$, but the
successive error, $\norm{u_{n+1} - u_n}_{H^1}$, does not follow
$\nu_2^n$.  Instead, it follows $\nu_3^n$, $\nu_3 \approx .38$ being
the next smallest eigenvalue.




This can be understood in terms of symmetry.  Problem
\eqref{e:boundstate_1d} is a symmetric problem with even solution, the
iteration scheme preserves symmetry, and the initial guess
\eqref{e:boundstate_1d_guess} is even.  Thus, we are restricted to the
closed subspace of even $H^1_0(-\Xmax,\Xmax)$ functions.  The
eigenvalue, $\nu_2$, corresponds to an odd eigenfunction, and thus
plays no role.

\subsection{Robustness to the Initial Condition}

In the preceding example, our initial guess
\eqref{e:boundstate_1d_guess} is a highly informed choice; it is
signed, smooth, and even.  This begs the question of how the algorithm
behaves when a poor guess is made.  Let us try run the algorithm with
a ``rough'' starting guess
\begin{equation}
  \label{e:rough1d}
  u_0 = \sum_{j} c_j \varphi_j, \quad c_j \sim U(-1,1),
\end{equation}
where $\varphi_j$ are the hat function finite elements, and $c_j$ are
independent, identically distributed, uniform random variables.  We
also consider the asymmetric, but smooth, initial condition
\begin{equation}
  u_0 = \tfrac{1}{2}(\Xmax-x)^2 (\Xmax +x),
\end{equation}
and the symmetrization of \eqref{e:rough1d}
\[
\frac{1}{2}\paren{u_0^{\text{rough}}(x) + u_0^{\text{rough}}(-x)}.
\]
These initial guesses, along with the solution are shown in Figure
\ref{f:rough_profiles1d}.  For all guesses, we obtain the solution,
and as shown in Figure \ref{f:rough_profiles1d_convergence}, the
algorithm is relatively insensitive to the roughness of the data.
Instead, it is the symmetry that dominates the problem.  These
computations reflect Theorem \ref{t:global} and Corollary \ref{c:combo}, which
tell us that, generically, a subsequential limit exists, and if this
limit is linearly stable, then the entire sequence converges to it.

\begin{figure}
  \subfigure{\includegraphics[width=6.35cm]{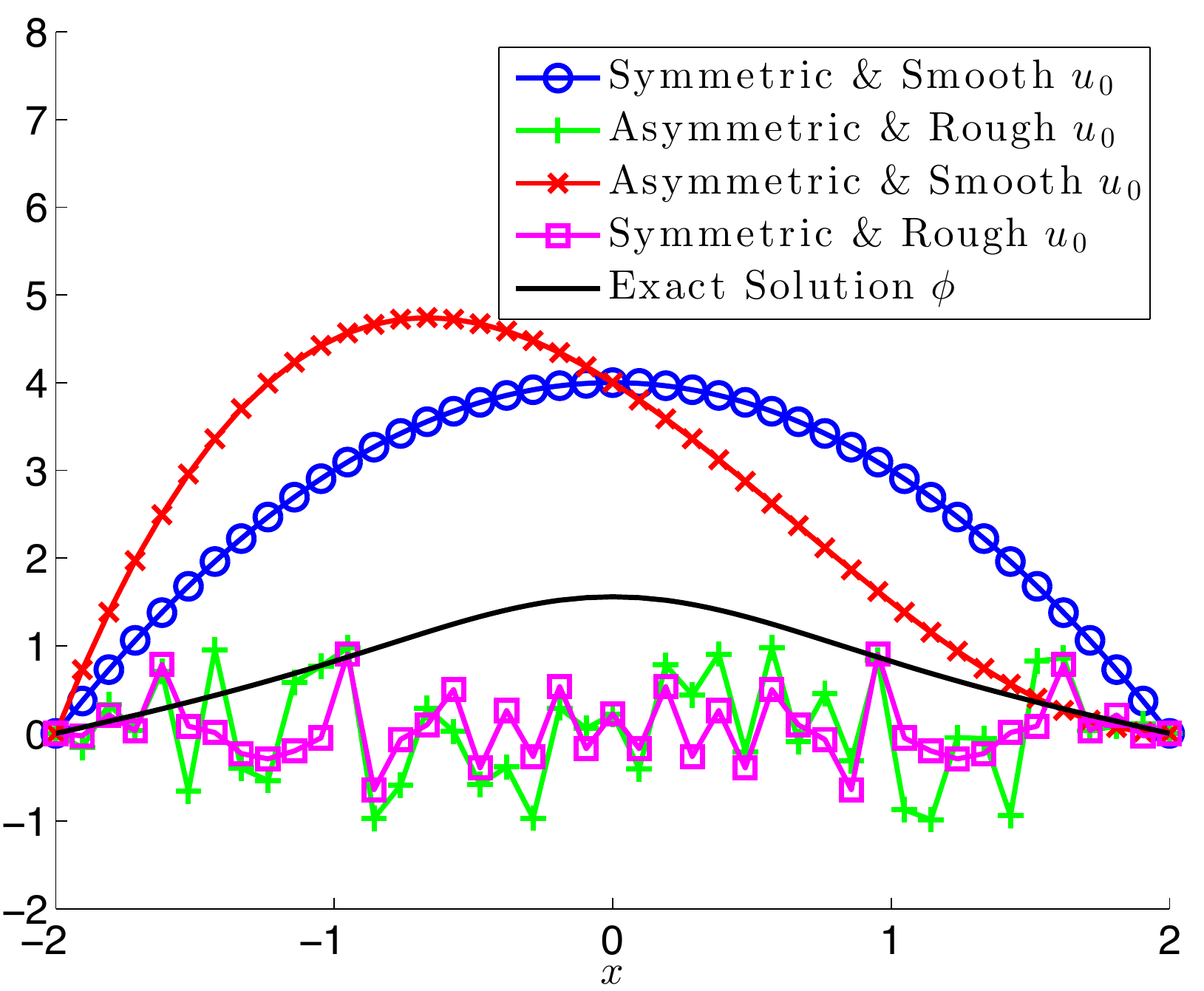} \label{f:rough_profiles1d}}
  \subfigure{\includegraphics[width=6.35cm]{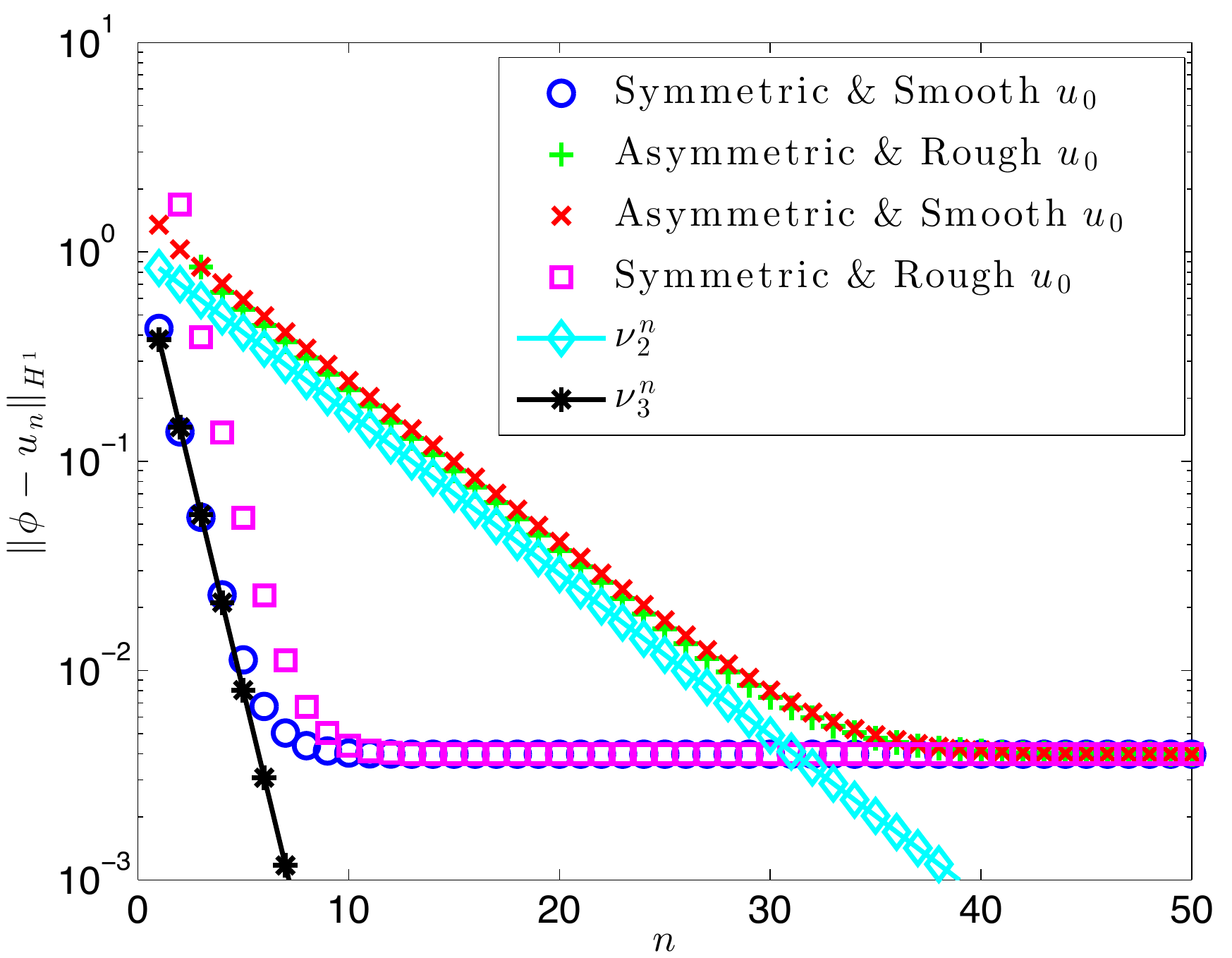} \label{f:rough_profiles1d_convergence}}

  \caption{Initial profiles used to solve
    \eqref{e:boundstate_1d_guess} and the convergence of the algorithm
    towards a solution. Symmetric initial conditions converge at a
    rate of $\nu_3< \nu_2$ while asymmetric ones converge at this
    latter, slower rate.}

\end{figure}

\subsection{2D Examples with Radial Symmetry}

Next, we consider solving the problem in 2D under a radial symmetry
assumption.  This reduces it to a 1D problem on $(0, \Rmax)$ with a
Neumann boundary condition at the origin:
\begin{equation}
  \label{e:radial}
  -\partial_r^2 \phi- \tfrac{1}{r} \partial_r\phi + \phi -
  \abs{\phi}^2\phi = 0, \quad \left.\partial_r\phi \right|_{r=0} =
  \phi(\Rmax) = 0.
\end{equation}
Starting from the initial guess,
\begin{equation}
  \label{e:radu0}
  u_0(r) = \Rmax^2 - r^2,
\end{equation}
we compute the ground state using piecewise linear finite elements on
the domain $(0, 25)$.  With five hundred uniformly spaced elements in
this interval, we obtain the ground state solution shown in Figure
\ref{f:2d_radial}.

\begin{figure}
  \subfigure[]{\includegraphics[width=6.45cm]{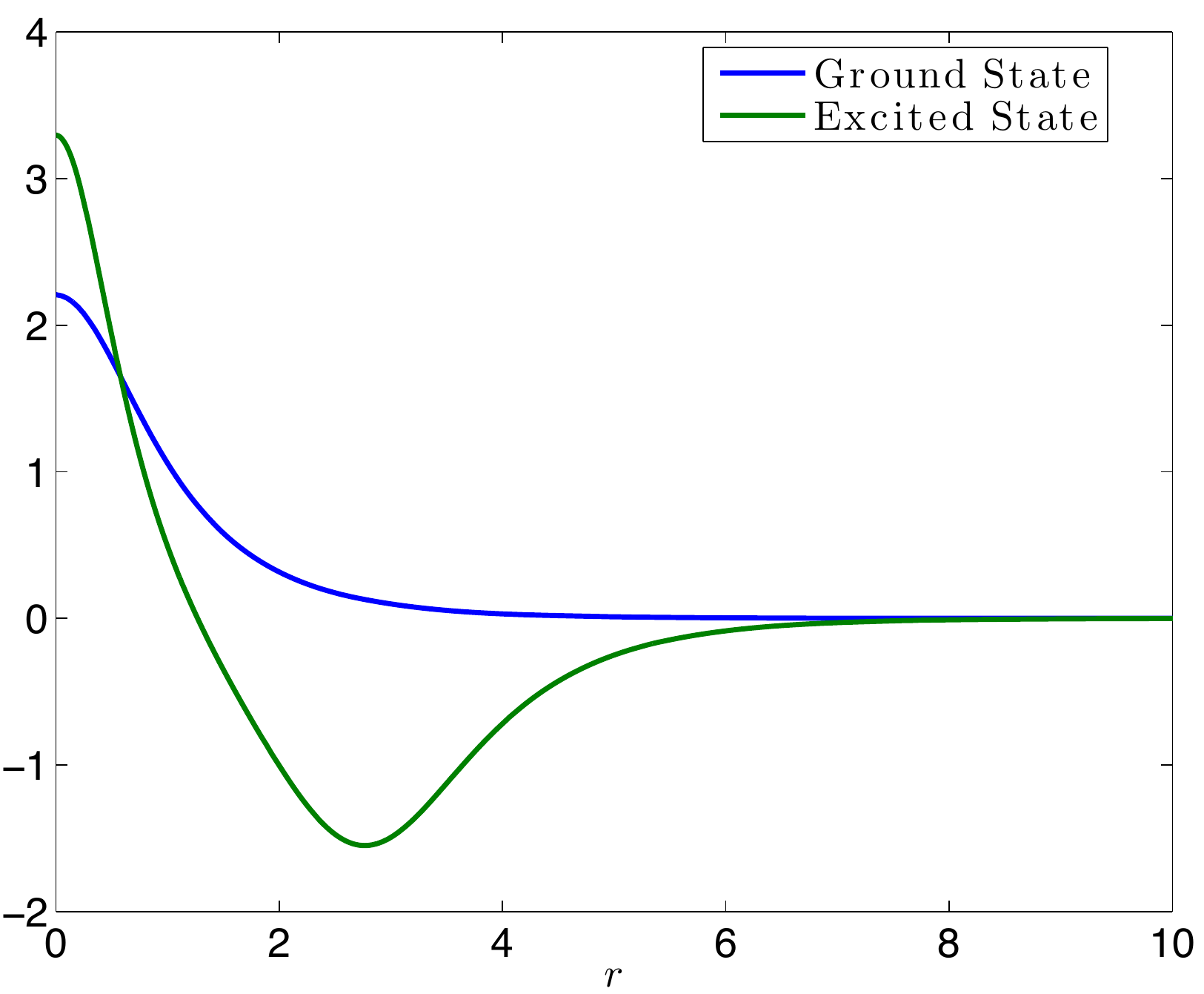}}
  \subfigure[]{\includegraphics[width=6.45cm]{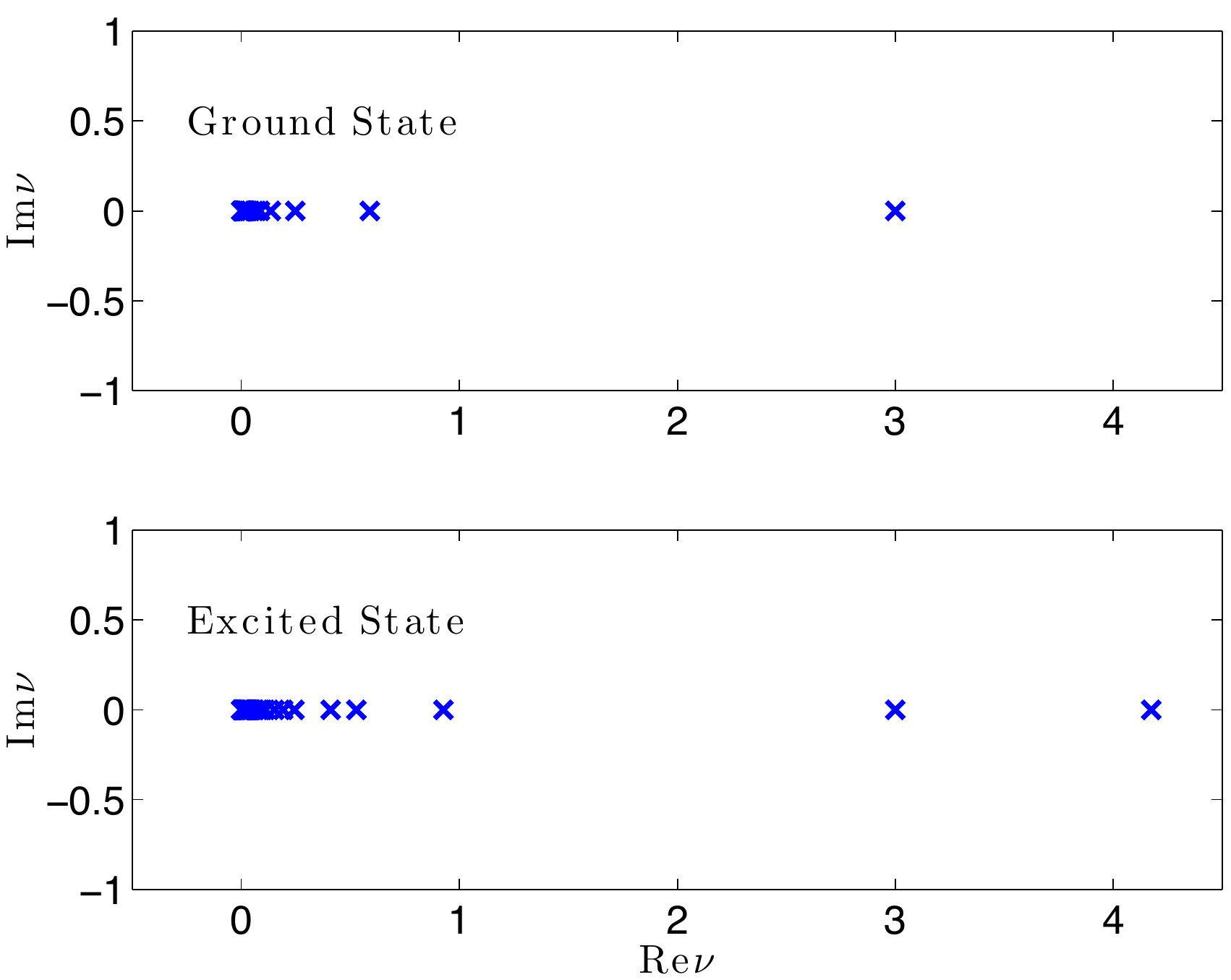}}
  \caption{The ground state and an excited state solution to
    \eqref{e:radial} computed on the domain $(0, 25)$, along with the
    spectrum of $\calL^{-1}(p\abs{\phi}^{p-1}\bullet)$ for each
    state. Both have the unstable eigenvalue at $p=3$, but the excited
    state has a second unstable eigenvalue. }
  \label{f:2d_radial}
\end{figure}

It is known that on the domain $(0, \infty)$, \eqref{e:radial} has
infinitely many real valued solutions with an arbitrary number of zero
crossings \cite{Sulem:1999kx}.  And yet, no matter the initial
condition $u_0$, the algorithm always appears to converge to a ground
state.  This begs the question of why.  Here, we adapt our algorithm
to compute an excited state solution, and, {\it a posteriori}, we
observe that it violates the spectral Assumption \ref{a:spectral},
introducing a linear instability.

Given $r_0 \in (0, \Rmax)$, let $\phi^{(0)}$ and $\phi^{(1)}$ solve
the two boundary value problems
\begin{gather}
  \label{e:radial1}
  \begin{split}
    -\partial_{r}^2\phi^{(0)}- \tfrac{1}{r}\partial_r\phi^{(0)} +
    \phi^{(0)} - |\phi^{(0)}|^2\phi^{(0)} = 0, \quad 0<r<r_0\\
    \left. \partial_r \phi^{(0)} \right|_{r=0} = \phi^{(0)}(r_0) = 0,
  \end{split}\\
  \label{e:radial2}
  \begin{split}
    -\partial_{r}^2\phi^{(1)}- \tfrac{1}{r}\partial_r\phi^{(1)} +
    \phi^{(1)} - |\phi^{(1)}|^2\phi^{(1)} = 0,\quad r_0<r<\Rmax\\
    \phi^{(1)}(r_0) = \phi^{(1)}(\Rmax)=0.
  \end{split}
\end{gather}
We then define the slope mismatch function as
\begin{equation}
  \label{e:slope_mistmatch}
  F(r_0) = \left. \partial_r \phi^{(0)}
  \right|_{r=r_0}-\left.\partial_r \phi^{(1)}\right|_{r=r_0}. 
\end{equation}
Wrapping a root finding algorithm around our solvers, we compute the
value of $r_0$ at which the slopes match, which yields the excited
state plotted in Figure \ref{f:2d_radial}.

Using these computed solutions, we then compute the linearized
spectrum plotted in Figure \ref{f:2d_radial}.  The excited state has a
second eigenvalue larger than one.  Thus, it is linearly unstable with
respect to Petviashvilli's method, and this is why it is not found
directly.  It remains to be studied why the root finding problem
associated with \eqref{e:slope_mistmatch} stabilizes the algorithm.

\subsection{2D Examples without Symmetry}

Finally, we demonstrate the algorithm on a domain that cannot be
reduced to 1D.  This is computed using FEniCS
\cite{LoggMardalEtAl2012a}.  The problem is studied using piecewise
linear triangular elements with parameters $p=3$ and
$\gamma = \gamma_\star$.

Consider the problem on the isosceles right triangle inscribed in
$[0,1]^2$.  Here, an unstructured mesh was generated using Triangle
\cite{shewchuk96b}.  As an initial guess, we take
\begin{equation}
  u_0 = \exp\set{-50 \bracket{ (x-\tfrac{1}{4})^2 + (y-\tfrac{1}{4})^2  }}.
\end{equation}
The solution, at several values of $n$, are displayed in Figure
\ref{f:tri_soln}, and the diagnostics, $\abs{M[u_n]-1}$ and
$\norm{u_{n+1} - u_n}_{H^1}$ are shown in Figure \ref{f:tri_diags}.
As in the 1D case, the algorithm rapidly enters the linear regime and
returns the ground state despite our choice of initial condition.

\begin{figure}
  \subfigure{\includegraphics[width=6.45cm]{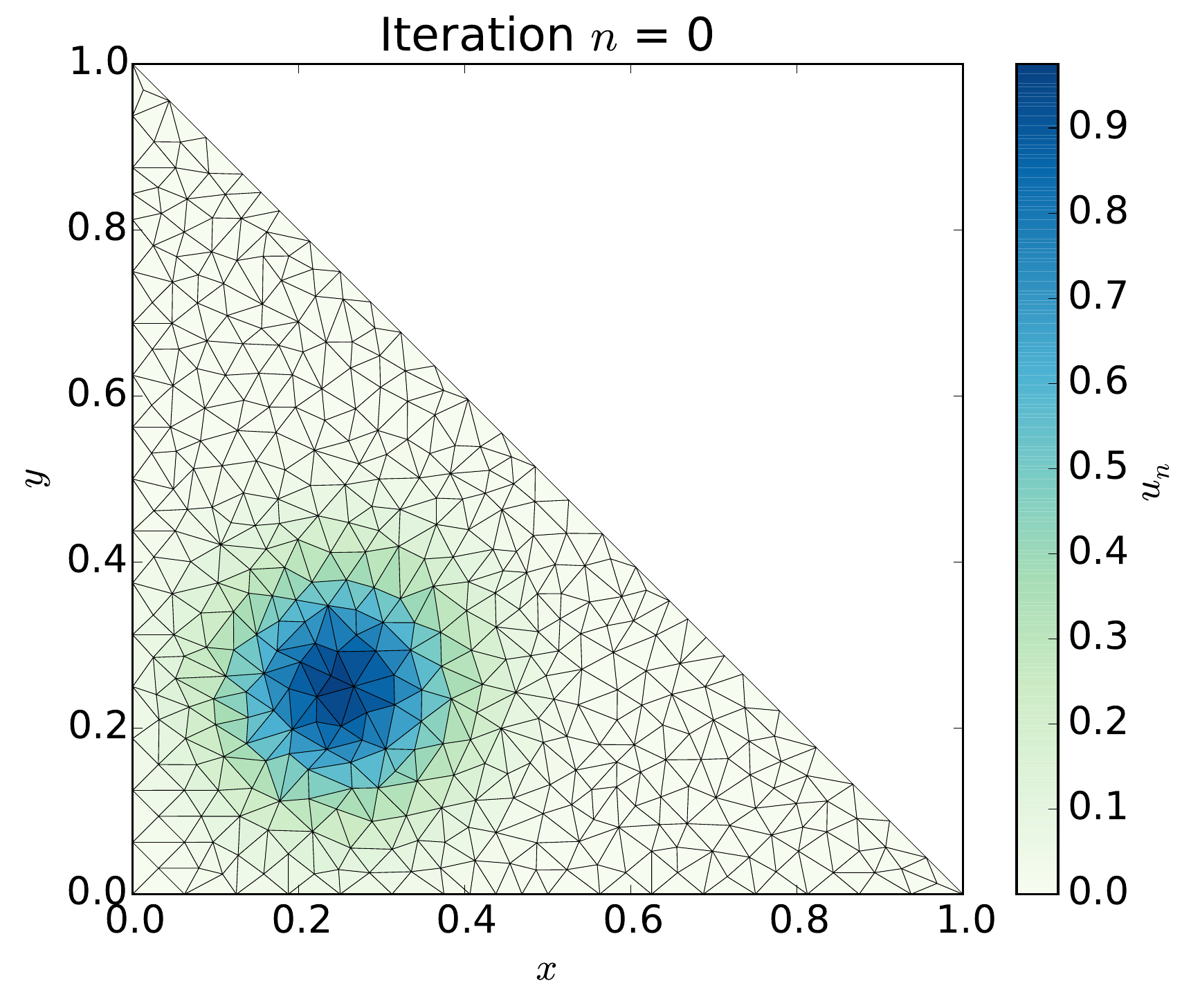}}
  \subfigure{\includegraphics[width=6.45cm]{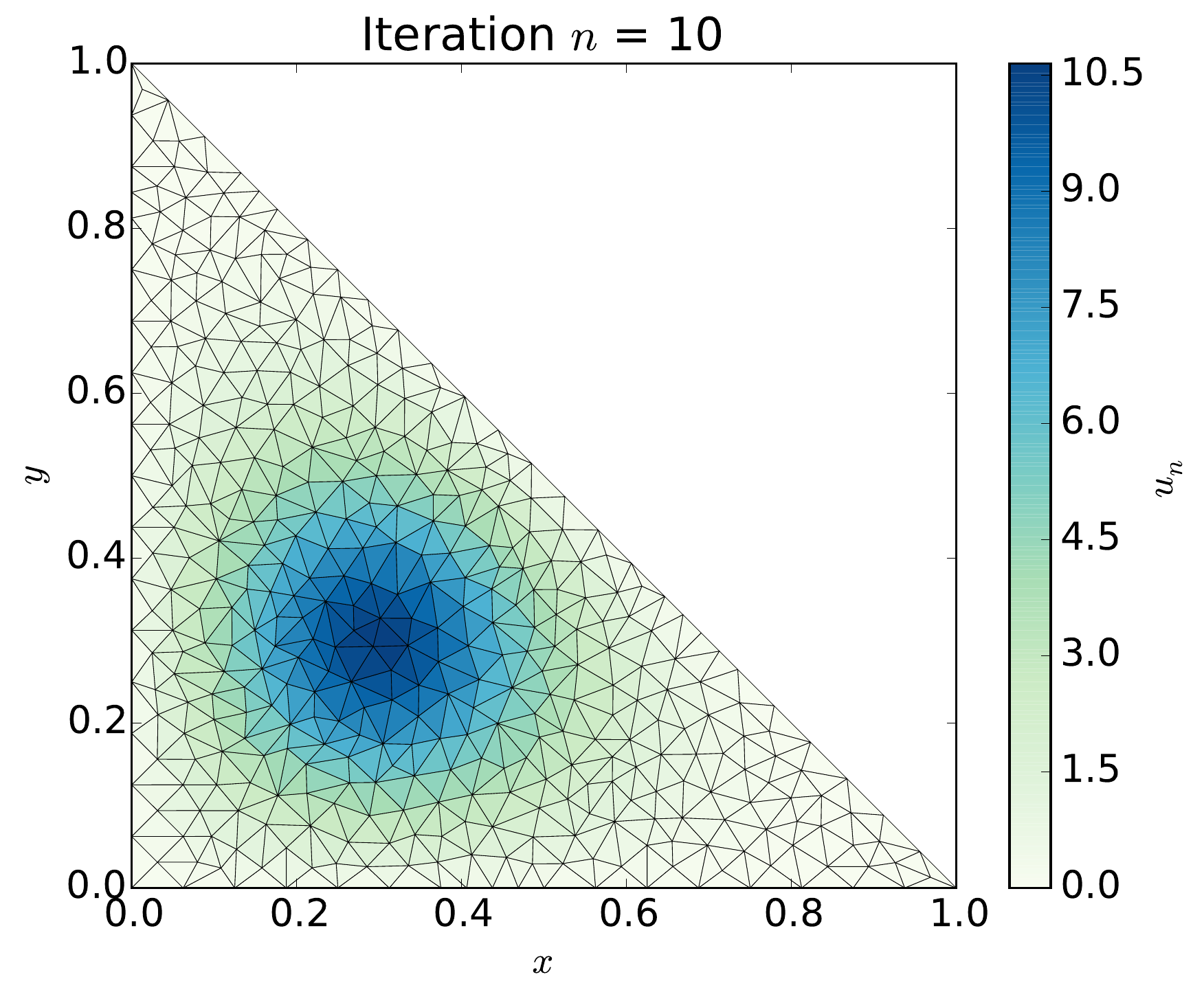}}

  \subfigure{\includegraphics[width=6.45cm]{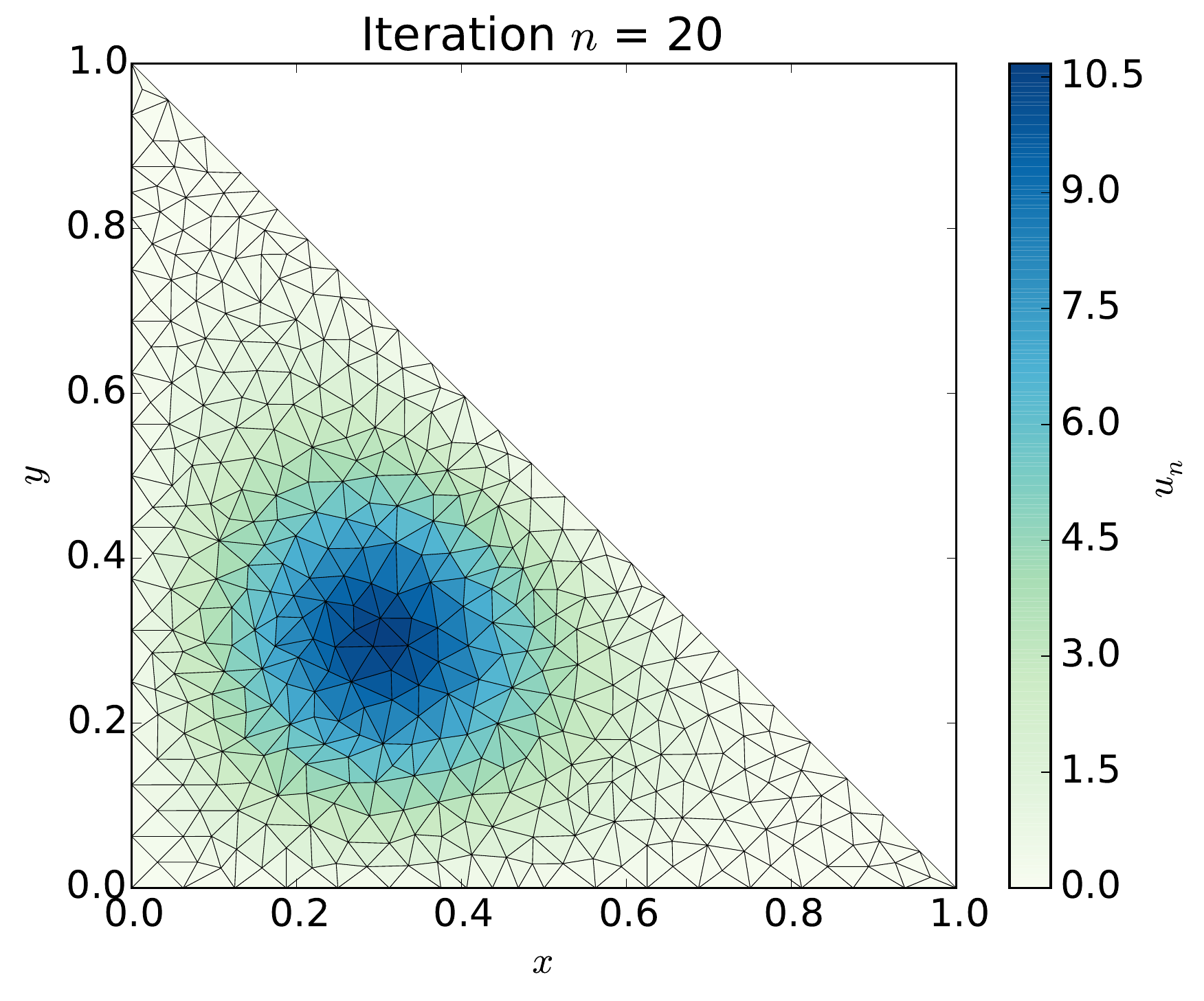}}
  \subfigure{\includegraphics[width=6.45cm]{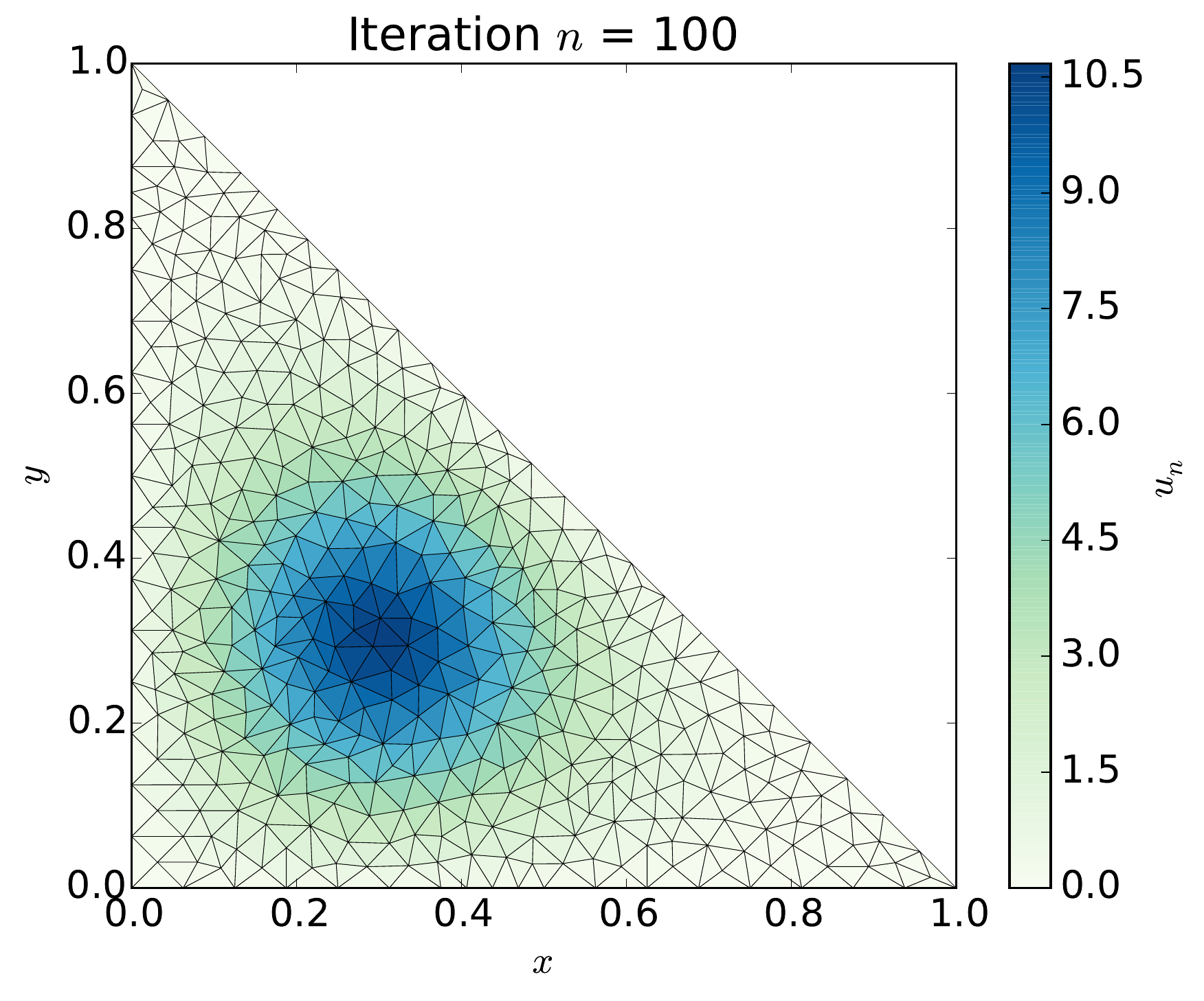}}
  \caption{Solution to \eqref{e:boundstate} on the triangular domain
    at several iterations.  Note the change in scales between
    iterations.}
  \label{f:tri_soln}
\end{figure}

\begin{figure}
  \begin{center}
    {\includegraphics[width=7cm]{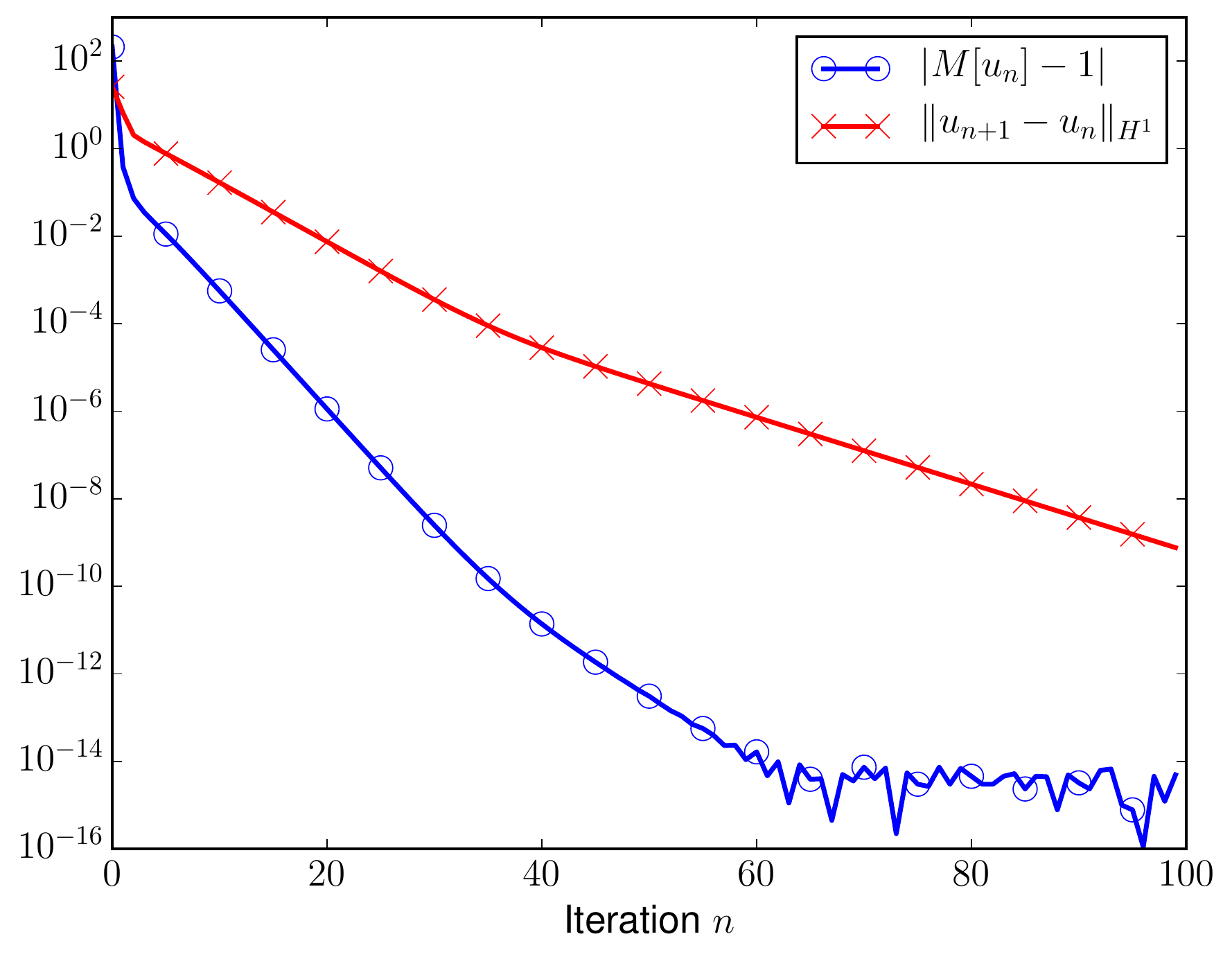}}
  \end{center}
  \caption{Diagnostics of the convergence of $u_n\to \phi$ when
    solving \eqref{e:boundstate} on the triangular domain.}
  \label{f:tri_diags}
\end{figure}

\section{Discussion}
\label{s:disc}

We have demonstrated and analyzed a computational algorithm for
solving the Dirichlet problem for \eqref{e:boundstate}.  Here, we
remark on some of our results and highlight open problems.

Our local results were predicated on Assumption \ref{a:spectral}, and
we have shown that, in some cases, it can be verified {\it a
  posteriori}.  But the question remains as to whether or not it will
hold generically for the ground state solution to the Dirichlet
problem.  For the problem on $\R^d$, more can been said.  In
\cite{Pelinovsky:2004bv}, it is shown that the number of unstable
directions of $(I-\Delta)^{-1}(p\abs{\phi}^{p-1}\bullet)$ can be
related to the number of negative eigenvalues of
\begin{equation*}
  \mathcal{H} =I - \Delta  - p\abs{\phi}^{p-1}\psi.
\end{equation*}
$\calH$ is just the linearization of \eqref{e:Rn_boundstate}, and its
spectrum is well characterized \cite{Sulem:1999kx}.  In the
subcritical regime, $(p-1)d < 4$, $\calH$ has a single negative
eigenvalue which will correspond to the unstable eigenvalue $p$ for
$(I-\Delta)^{-1}(p\abs{\phi}^{p-1}\bullet)$. $\calH$ also has a
kernel, $\nabla \phi$, which would correspond to a neutral eigenvalue
of 1 in the iteration algorithm.

For the Dirichlet problem, $\nabla \phi$ is not in $H^1_0$, as it will
fail to vanish at the boundary, so the neutral mode is eliminated.  We
conjecture that on sufficiently large and symmetric domains, the
Dirichlet problem will inherit the same number of unstable eigenvalues
as that of the problem on $\R^d$, but for small domains, or domains
with significant geometry, we cannot say, {\it a priori}, what becomes
the spectrum.

Our approach can be applied to simple generalizations of
\eqref{e:boundstate} such as
\begin{equation*}
  -\lambda \phi + \Delta \phi + \abs{\phi}^{p-1} \phi = 0,
\end{equation*}
where $\lambda$ is assumed larger than the first Dirichlet eigenvalue
of the Laplacian on $\Omega$.  More significant generalizations would
involve the inclusion of external potentials or symmetry breaking
nonlinearities.  This analysis might also be carried out for problems
with other boundary conditions, such as Neumann or Robin.

The main outstanding problem remains to explain the experimentally observed generic convergence towards ground state
solutions, even for data far from any solution.  We have
shown strong subsequential convergence, but we can only show that the sequence
is Cauchy when the limiting solution satisfies particular spectral
properties.  

Another challenge is to consider the total error of the algorithm, as
an approximation to the problem on $\R^d$.  The total error could be
decomposed as:
\begin{equation*}
  \begin{split}
    \underbrace{\norm{\phi - u_n^{(\Delta x, \Omega)}}}_{\text{Total
        Error}} &\leq \underbrace{\norm{\phi -
        \phi^{(\Omega)}}}_{\text{Modeling Error}}+
    \underbrace{\norm{\phi^{(\Omega)} -\phi^{(\Delta
          x,\Omega)}}}_{\text{Discretization Error}} +
    \underbrace{\norm{\phi^{(\Delta x,\Omega)}- u_n^{(\Delta
          x,\Omega)}}}_{\text{Algorithmic Error}}.
  \end{split}
\end{equation*}
We expect that the truncation error will be most severe with the
Dirichlet boundary condition, which is why Robin boundary conditions
merit study.  A particularly important case for such a complete
analysis would be the radial case in dimension $d$.  This is because
investigations of soliton dynamics for \eqref{e:nls} and \eqref{e:nlw}
are frequently performed with radial symmetry.  modelling
\section*{Acknowledgements}

\noindent The authors are grateful for several helpful conversations
with Svitlana Mayboroda.  D.Olson was supported by the Department of
Defense (DoD) through the National Defense Science and Engineering
Graduate Fellowship (NDSEG) Program.  S.Shukla was supported by
University of Minnesota UROP-11133.  G.Simpson began this work under
the support of the DOE DE-SC0002085 and the NSF PIRE OISE-0967140, and
completed it under NSF DMS-1409018.  D.Spirn was supported by NSF
DMS-0955687.

\bibliographystyle{siam}

\bibliography{pet_refs}

\begin{thebibliography}{10}

\bibitem{Ablowitz:2005tu}
{\sc M.~J. Ablowitz and Z.~H. Musslimani}, {\em {Spectral renormalization
  method for computing self-localized solutions to nonlinear systems}}, Opt.
  Lett., 30 (2005), pp.~2140--2142.

\bibitem{adolfsson:1993}
{\sc V.~Adolfsson}, {\em {$L^p$-integrability of the second order derivatives
  of Green potentials in convex domains}}, Pac. J. Math., 159 (1993),
  pp.~201--225.

\bibitem{Alvarez:2014ec}
{\sc J.~{\'A}lvarez and A.~Duran}, {\em {Petviashvili type methods for
  traveling wave computations: I. Analysis of convergence}}, J. Comput. Appl.
  Math., 266 (2014), pp.~39--51.

\bibitem{Badiale:2011ug}
{\sc M.~Badiale and E.~Serra}, {\em {Semilinear elliptic equations for
  beginners, existence results via the variational approach}}, Springer, 2011.

\bibitem{Baruch:2011fm}
{\sc G.~Baruch and G.~Fibich}, {\em {Singular solutions of the
  $L^{2}$-supercritical biharmonic nonlinear Schr\"odinger equation}},
  Nonlinearity, 24 (2011), pp.~1843--1859.

\bibitem{Cazenave:aa}
{\sc T.~Cazenave}, {\em An introduction to semilinear elliptic equations},
  Federal University of Rio, Editora do IM-UFRJ, Rio de Janeiro, 2006.
\newblock \url{https://www.ljll.math.upmc.fr/cazenave/77.pdf}.

\bibitem{Chugunova:2007wg}
{\sc M.~Chugunova and D.~E. Pelinovsky}, {\em {Two-pulse solutions in the
  fifth-order KdV equation}}, Discrete Cont. Dyn.-B, 8 (2007), pp.~773--800.

\bibitem{Demanet:2006gi}
{\sc L.~Demanet and W.~Schlag}, {\em {Numerical verification of a gap condition
  for a linearized nonlinear Schr{\"o}dinger equation}}, Nonlinearity, 19
  (2006), pp.~829--852.

\bibitem{fromm:1993}
{\sc S.~Fromm}, {\em Potential space estimates for green potentials in convex
  domains}, P. Am. Math. Soc., 119 (1993), pp.~225--233.

\bibitem{GilbargTrudinger}
{\sc D.~Gilbarg and N.~S. Trudinger}, {\em Elliptic partial differential
  equations of second order}, Springer-Verlag, 2001.

\bibitem{Hutson:2005aa}
{\sc V.~Hutson, J.~S. Pym, and M.~J. Cloud}, {\em Applications of Functional
  Analysis and Operator Theory Analysis and Operator Theory}, Elsevier,
  second~ed., 2005.

\bibitem{Lakoba:2007cg}
{\sc T.~I. Lakoba and J.~Yang}, {\em {A generalized Petviashvili iteration
  method for scalar and vector Hamiltonian equations with arbitrary form of
  nonlinearity}}, J. Comput. Phys., 226 (2007), pp.~1668--1692.

\bibitem{Lakoba:2007kp}
\leavevmode\vrule height 2pt depth -1.6pt width 23pt, {\em {A mode elimination
  technique to improve convergence of iteration methods for finding solitary
  waves}}, J. Comput. Phys., 226 (2007), pp.~1693--1709.

\bibitem{LoggMardalEtAl2012a}
{\sc A.~Logg, K.-A. Mardal, G.~N. Wells, et~al.}, {\em Automated Solution of
  Differential Equations by the Finite Element Method}, Springer, 2012.

\bibitem{Musslimani:2004wx}
{\sc Z.~H. Musslimani and J.~Yang}, {\em {Self-trapping of light in a
  two-dimensional photonic lattice}}, J. Opt. Soc. Am. B, 21 (2004),
  pp.~973--981.

\bibitem{Pelinovsky:2004bv}
{\sc Dmitry~E Pelinovsky and Yury~A Stepanyants}, {\em {Convergence of
  Petviashvili's iteration method for numerical approximation of stationary
  solutions of nonlinear wave equations}}, SIAM J. Numer. Anal., 42 (2004),
  pp.~1110--1127.

\bibitem{Petviashvilli:1976aa}
{\sc V.~I. Petviashvilli}, {\em Equation of an extraordinary soliton}, Plasma
  Physics, 2 (1976).

\bibitem{Reed:1980aa}
{\sc M.~Reed and B.~Simon}, {\em Methods of Modern Mathematical Physics:
  Functional Analysis}, vol.~1, Academic Press, 1980.

\bibitem{shewchuk96b}
{\sc J.~R. Shewchuk}, {\em Triangle: {E}ngineering a {2D} {Q}uality {M}esh
  {G}enerator and {D}elaunay {T}riangulator}, in Applied Computational
  Geometry: Towards Geometric Engineering, Ming~C. Lin and Dinesh Manocha,
  eds., Springer-Verlag, 1996, pp.~203--222.

\bibitem{Sulem:1999kx}
{\sc C.~Sulem and P.-L. Sulem}, {\em {The Nonlinear Schr{\"o}dinger Equation:
  Self-Focusing and Wave Collapse}}, Springer, 1999.

\bibitem{Yang:2007vu}
{\sc J.~Yang and T.~I. Lakoba}, {\em {Universally-convergent squared-operator
  iteration methods for solitary waves in general nonlinear wave equations}},
  Stud. Appl. Math., 118 (2007), pp.~153--197.

\bibitem{Yang:2008kn}
\leavevmode\vrule height 2pt depth -1.6pt width 23pt, {\em {Accelerated
  imaginary-time evolution methods for the computation of solitary waves}},
  Stud. Appl. Math., 120 (2008), pp.~265--292.

\end{thebibliography}

\end{document}